\documentclass[11pt]{article}

\usepackage{amsmath,amssymb}
\usepackage{amsthm}
\usepackage[nobysame]{amsrefs}
\usepackage{indentfirst}
\usepackage{mathrsfs}
\usepackage{graphicx}
\usepackage{epstopdf}
\usepackage{algorithm}
\usepackage{algorithmic}
\usepackage{xcolor}
\usepackage{fourier}
\usepackage[margin=1.0in]{geometry}
\usepackage[norelsize,boxed,linesnumbered,vlined,algo2e]{algorithm2e} 

\usepackage{mathtools}
\makeatletter
\DeclareRobustCommand*\cal{\@fontswitch\relax\mathcal}
\makeatother

\usepackage{booktabs}

\usepackage{latexsym,enumerate}
\usepackage{multirow}

\def\bb{\mathbb}

\newcommand{\comment}[1]{}

\newcommand{\EEA}{\end{eqnarray}}
\newcommand{\BEA}{\begin{eqnarray}}

\newcommand{\peq}{p_{eq}}

\newcommand{\td}{\textrm{d}}
\newcommand{\OL}{\mathcal{L}}

\newcommand{\rank}{\operatorname{rank}}
\newcommand{\spn}{\operatorname{span}}
\newcommand{\Var}{\operatorname{Var}}

\newtheorem{thm}{Theorem}[section]
\newtheorem{prop}[thm]{Proposition}
\newtheorem{lem}[thm]{Lemma}

\numberwithin{thm}{section}

\newtheorem{defn}[thm]{Definition}

\newtheorem{rem}[thm]{Remark}

\begin{document}

\title{A Parameter Estimation Method Using Linear Response Statistics: Numerical Scheme}

\author{He Zhang$^{1}$\footnote{hqz5159@psu.edu}, Xiantao Li$^{1}$\footnote{xxl12@psu.edu}, and John Harlim$^{1,2}$\footnote{jharlim@psu.edu} \\
  \vspace{0.1in}\\
  $^{1}$ Department of Mathematics,\\ The Pennsylvania State University, University Park, PA 16802, USA.\\
  $^{2}$ Department of Meteorology and Atmospheric Science, \\ Institute for CyberScience \\ The Pennsylvania State University, University Park, PA 16802, USA.\\
}

\maketitle

\begin{abstract}
This paper presents a numerical method to implement the parameter estimation method using response statistics that was recently formulated by the authors. The proposed approach formulates the parameter estimation problem of It\^o drift diffusions as a nonlinear least-squares problem. To avoid solving the model repeatedly when using an iterative scheme in solving the resulting least-squares problems, a polynomial surrogate model is employed on appropriate response statistics with smooth dependence on the parameters. The existence of minimizers of the approximate polynomial least-squares problems that converge to the solution of the true least square problem is established under appropriate regularity assumption of the essential statistics as functions of parameters. Numerical implementation of the proposed method is conducted on two prototypical examples that belong to classes of models with wide range of applications, including the Langevin dynamics and the stochastically forced gradient flows. Several important practical issues, such as the selection of the appropriate response operator to ensure the identifiability of the parameters and the reduction of the parameter space, are discussed. From the numerical experiments, it is found that the proposed approach is superior compared to the conventional approach that uses equilibrium statistics to determine the parameters.

\textbf{Keywords}: Inverse Problem, Linear Response Theory, Surrogate Model, Nonlinear Least-Square.
\end{abstract}

\maketitle

{\bf
\bf In many dynamical systems, some of the model parameters may not be identifiable from their equilibrium statistics. For problems where the parameters can be identified from their two-point statistics, this paper provides a method to numerically estimate these parameters. Mathematically, we formulate the inverse problem as a well-posed dynamic constrained nonlinear least-squares problem that fits a set of linear response statistics under small external perturbations. Since solving this problem directly is computationally not feasible, we consider a polynomial surrogate modeling to approximate this least-squares problem. Convergence of the solution of the approximate problem to the solution of the original nonlinear least-squares solution is established under appropriate regularity assumptions. Supporting numerical examples on two classes of models with a wide range of applications, the Langevin dynamics and the stochastically forced gradient flows, are given. 
}

\section{Introduction}
Parameter estimation is ubiquitous in modeling of nature. The goal in this inverse problem is to infer the parameters from observations with sufficient accuracy so that the resulting model becomes a useful predictive tool. Existing parameter estimation techniques often belong to one of the following two classes of approaches, the maximum likelihood estimation \cite{Pavliotis:16} and the Bayesian inference such as  Markov chain Monte-Carlo (MCMC) \cite{gamerman:06}, depending on the availability of prior information about the parameters. Since the model parameters are usually not directly measured, the success of any inference method depends crucially on the identifiability of the parameters from the given observations. When the dependence of the observations on the parameters are implicit, that is, through the model, sensitivity analysis (see e.g. the review article \cite{bp:16}) is often a useful practical tool to determine the parameter identifiability. 

In this paper, we consider the parameter estimation of ergodic stochastic differential equations driven by Brownian noise. When the corresponding invariant measure of the dynamical systems have an explicit dependence on the parameters, then these parameters can usually be inferred from appropriate equilibrium statistical moments. A popular approach that exploits this idea is the reverse Monte Carlo method \cite{lyubartsev1995calculation} developed in Chemistry. In particular, it formulates an appropriate nonlinear least-squares system of integral equations by matching the  equilibrium averages of some pre-selected observables. Subsequently, Newton's iterations are used to estimate the parameters. At each iterative step, samples at the current parameter estimate are generated for Monte-Carlo estimation to approximate the least-squares integral equations and the corresponding Jacobian matrix. In practice, this method can be rather slow due to the repeated sampling procedure.   
A severe limitation of this approach is that it is restrictive to inference of parameters of the equilibrium density.

This limitation can be overcome by fitting appropriate two-point statistics. As shown in our previous work \cite{HLZ:17}, we formulated the parameter estimation problem based on the linear response statistics subjected to an external forcing, which drives the system out of equilibrium. The fluctuation-dissipation theory (FDT), a hallmark in non-equilibrium statistical physics, suggested that the changes of the average of an observable under small perturbations can be approximated by appropriate two-point statistics, called the FDT response operators \cite{Toda-Kubo-2}. The key point is that these FDT response operators can be estimated using the available samples of the equilibrium {\it unperturbed} dynamics so long as we know the exact form of the invariant measure of the dynamics, which we will assume to be the case for a large class of problems, such as the Langevin dynamics and stochastic gradient flows. We should point out that the proposed approach relies on the validity of FDT response statistics, which has been studied rigorously for a large class of stochastic system that includes the stochastic differential equation (SDE) setup in this paper \cite{hairer2010simple}. For deterministic dynamics, one can use for example the statistical technique introduced in \cite{gww:2016} to verify the validity of the FDT linear response. Following the ideas from \cite{qm:16a,mq:16a,mq:17}, we developed a parameter inference method using these response operators in \cite{HLZ:17}. While the method in \cite{qm:16a,mq:16a,mq:17} involves minimizing an information-theoretic functional that depends on both the mean and variance response operators, our approach fits a finite number of quantities, to be referred
to as the {\it essential statistics}, which would allow us to approximate the FDT response operators, of appropriate observables (beyond just the mean and variance). 

In our previous work \cite{HLZ:17}, we showed the well-posedness of the formulation on { three examples}, in the sense that under infinite sampling, appropriate choices of essential statistics will ensure the identifiability of the model parameters. {  On one of these examples, the simplified model for turbulence \cite{majda2016}, which is a nonlinear system of SDEs with a Gaussian invariant measure, we were able to explicitly determine the choice of observables and external forcings (which in turn determine the FDT response operators) that allow one to identify all of the model parameters uniquely. In this case, one can directly estimate the parameters by solving the system of equations involving these essential statistics. While this result suggests that the choice of observables and external forcings is problem dependent, the FDT theory provides a guideline for choosing the appropriate two-point statistics for a well-posed parameter estimation. On another example, the Langevin dynamics (which will be discussed in Section~\ref{sec4}), while the parameters can be identified by the essential statistics as shown in \cite{HLZ:17}, the formulation suggests that one should include essential statistics that involve higher-order derivatives of the FDT response operators.  This can be problematic because these higher-order statistics are rarely available in practice unless the available data are solutions of high-order SDE solvers. Given this practical issue, we propose to fit only the essential statistics that include the zeroth- and first-order derivatives of the FDT response operator  ( and if possible, to use only the zeroth-order derivative information). The price paid by restricting to only the lower-order essential statistical information is that the dependence of the essential statistics on some of the model parameters becomes implicit through the solutions of the dynamical model and the identifiability of these parameters becomes questionable.}
 
{  In this paper, we devise a concrete numerical method to estimate the parameters by solving dynamic-constrained least-squares problems of integral equations involving these low-order essential statistics.} As one would imagine, the implementation of this approach will face several challenges. First, just like in the reverse Monte-Carlo method, naive implementation of an iterative method for solving this dynamically constrained least-squares problem requires solving the model repeatedly. To avoid solving the model repeatedly in the minimization steps, we employ a polynomial surrogate model approach \cite{Marzouk:07,Marzouk:09} on the least-squares cost function that involves the essential statistics. This approach is motivated by the fact that the cost function (or effectively the essential statistics) is subjected to sampling error and that the essential statistics have smooth dependence on the parameters, assuming that the Fokker-Planck operator of the underlying dynamics has smooth dependence on the parameters. Under appropriate regularity assumption on the essential statistics, we will provide a theoretical guarantee for the existence of minimizers of the approximate polynomial least-squares problem that converge to the solution of the true least-squares problem. { The second related issue, as mentioned above, is that certain (lower-order) essential statistics might not be sensitive enough to some of the parameters, which will lead to inaccurate estimations.} To ensure the practical identifiability of the parameters, we employ an empirical a priori sensitivity analysis based on the training data used in constructing the polynomial surrogate models and a posteriori local sensitivity analysis to ensure the validity of the a prior sensitivity analysis. 

{  While the proposed technique can be used to infer all of the model parameters, in practice, one should use the polynomial approximation to infer parameters that cannot be estimated directly to avoid the curse of dimensionality. This issue is a consequence of using the polynomial expansion in approximating the least-squares problem, where the required training data set for  constructing the surrogate model increases exponentially as a function of the parameter dimension. In the numerical examples below, we will apply the approximate least-squares problems to estimate parameters that cannot be directly estimated from matching equilibrium and/or two-point statistics that involve the first-order derivative of the FDT response operator.}

The rest of the paper is organized as follows. In Section \ref{sec2}, we briefly review the concept of essential statistics and parameter inference method using the linear response statistics developed in \cite{HLZ:17}. In Section \ref{sec3}, we present the proposed numerical algorithm based on the polynomial based surrogate model, discuss its convergence (with detailed proof in the Appendices). In Sections~\ref{sec4} and \ref{sec5}, we show applications on two nonlinear examples, a Langevin model and a stochastic gradient system with a triple-well potential, respectively. In Section \ref{sec6}, we conclude the paper with a summary and discussion.

\section{A Review of Essential Statistics}\label{sec2}
We begin by reviewing the parameter estimation formulation introduced in our previous paper \cite{HLZ:17}. Consider an $n-$dimensional It\^o diffusion
\begin{equation}\label{true_model}
\td X= b(X;\theta)\td t+ \sigma(X;\theta) \td W_{t},
\end{equation}
where the vector field $b(X;\theta)$ denotes the drift, and $\sigma(X;\theta)$ is the diffusion tensor. Both coefficients are assumed to be smooth functions in $X$ and $\theta$; $W_{t}$ represents the standard Wiener process, and the variable $\theta\in D$ contains model parameters where $D\subset \bb{R}^{N}$ is the parameter domain. In this paper we only consider the case where $D$ is bounded and for simplicity we assume $D=[-1,1]^{N}$. Throughout this manuscript, we denote the underlying true parameter value as $\theta^{\dagger}$ and the corresponding estimate as $\hat{\theta}$. We assume that equation (\ref{true_model}) is ergodic (see \cite{Mattingly:02} for precise conditions) with equilibrium density $p_{eq}(x;\theta)$ for all $\theta\in D$, and $p_{eq}(x;\theta)$ smoothly depends on both $x$ and $\theta$. When $\theta=\theta^{\dagger}$ we assume that we have the access to the explicit formula of $p_{eq}(x;\theta^{\dagger})$ as a function of $x$ only, which is shortened as $p_{eq}^{\dagger}(x)$. As we shall see in many applications, $p_{eq}(x;\theta)$ only depends on a subset of the parameters $\theta$.

The main idea of the parameter estimation formulation introduced in \cite{HLZ:17} is to infer the parameters from the linear response operator associated with the fluctuation dissipation theory (FDT). Recall that, in the FDT setting we are interested in an order-$\delta$ ($0<\delta\ll 1$) external perturbation of the form $f(x,t) = c(x)\delta f(t)$, such that the solutions of the perturbed dynamics
\begin{eqnarray}
{  \td X^\delta = \big(b(X^\delta,\theta) + c(X^\delta)\delta f(t)\big)\td t +\sigma(X^\delta,\theta)\td W_t,\nonumber}
\end{eqnarray}
which would otherwise remain at the equilibrium state of the unperturbed dynamics \eqref{true_model}, can be characterized by a perturbed density $p^\delta(x,t;\theta)$. The density function is governed by the corresponding Fokker-Planck equation under initial condition $p^{\delta}(x,0;\theta)=p_{eq}(x;\theta)$.

By a standard perturbation technique, e.g., \cite{Pavliotis:16}, the difference between the perturbed and unperturbed statistics of any integrable function $A(x)$ can be estimated by a convolution integral, that is,
\begin{eqnarray}
\Delta \mathbb{E}[A](t):= \mathbb{E}_{p^\delta} [A(X)](t) - \mathbb{E}_{\peq} [A(X)] =  \int_0^t k_A(t-s)\delta f(s)\,ds + \mathcal{O}(\delta^2).\label{responsestat}
\end{eqnarray}
In \eqref{responsestat}, the term $k_A(t)$ is known as the linear response operator. The FDT formulates the linear response operator as the following two-point statistics
\begin{equation}\label{RA}
k_A(t;\theta):= \mathbb{E}_{\peq}[A(X(t))\otimes B(X(0);\theta)], \;\mbox{ with }\; B_{i}(X;\theta) := -\frac{\partial_{X_i} (c_i(X)\peq(X;\theta))}{\peq(X;\theta)},
\end{equation}
where $B_{i}$ and  $c_i$ denote the $i^{\text{th}}$ components of $B$ and $c$, respectively. We should point out that the validity of the FDT has been studied rigorously under mild conditions \cite{hairer2010simple}.  A more explicit form is given by
\begin{equation}\label{explicit_form}
k_A(t;\theta)=\int_{\bb{R}^{n}}\int_{\bb{R}^{n}} A(x)\otimes B(y) \rho(x,t|y,0)p_{eq}(y;\theta)\td x\td y, 
\end{equation}
where $\rho$ is the solution of the Fokker-Planck equation
\begin{equation}\label{fpe}
\frac{\partial}{\partial t} \rho=\OL^{*} \rho, \quad \rho(x,0|y,0)=\delta(x-y).
\end{equation}
Here $\OL$ denotes the generator of the unperturbed dynamics \eqref{true_model}, and $\OL^{*}$ is its adjoint operator in the $L^{2}(\mathbb{R}^n)$.

\begin{rem}\label{smooth_dep}
The result in \cite{Singler:08} states that if the coefficients of linear parabolic PDE's are $C^k$ as functions of $\theta$, then the weak solutions are also $k$-time differentiable with respect to $\theta$. In our case, if $b,\sigma\in C^k([-1,1]^N)$, then the linear response operator $k_{A}(t;\theta)$ \eqref{explicit_form} is also $C^{k}$ with respect to the parameters $\theta$ under the mild assumption that $p_{eq}(x;\theta)$ is smooth with respect to $\theta$. In Section~\ref{subsec:con}, we will conduct a convergence analysis on the proposed parameter estimation method to determine the necessary value of $k$, the smoothness of the linear response operators as functions of $\theta$.
\end{rem}

Notice that $B^{\dagger}(X):=B(X;\theta^{\dagger})$ can be determined analytically based on our assumption of knowing the explicit formula of $p^{\dagger}_{eq}$. Given $t$, the value of $k_{A}(t;\theta^{\dagger})$ can be computed using a Monte-Carlo sum based on the time series of the unpertubed system, $X$, at $p_{eq}^{\dagger}(x)$. Therefore, the linear response operator can be estimated without knowing the underlying unperturbed system in \eqref{true_model} so long as the time series of $X$ at $p_{eq}^{\dagger}(x)$ and the explicit formula for this equilibrium density is known.

The FDT, along with the expression of the kernel function \eqref{RA}, has been implemented to {\it predict} the change of the expectation of the observable $A$ \cite{leith:75,gbm:08,am:08,am:09,mw:10,am:12}. The advantage is that the response operator is defined with respect to the equilibrium distribution, which does not involve the prior knowledge of the perturbed density $p^{\delta}$. In statistical physics, the FDT is a crucial formulation to derive transport coefficients, e.g., electron conductivity, viscosity, diffusion coefficients, etc \cite{Toda-Kubo-2}. In the current paper, however, we propose to use the response statistics to infer the underlying true parameter value $\theta^{\dagger}$.

Since this linear response operator is in principle infinite-dimensional, it is necessary for practical purposes to introduce a finite-dimensional approximation. The following parametric form of the response operator has been motivated by the rational approximation of the Laplace transform \cite{Ma:16}. To explain the idea, consider the Laplace transformation of $k_{A}(t;\theta)$ denoted by $K(s;\theta)$, which can be approximated by a rational function in the following form
\begin{eqnarray}
K(s;\theta)\approx (I-s^{-1}\beta_{1}-\cdots- s^{-m}\beta_{m})^{-1}(s^{-1}\alpha_{1}+\cdots +s^{-m}\alpha_{m}), \quad \alpha_i, \beta_i \in \mathbb{R}^{n\times n}. \nonumber
\end{eqnarray}
In the time domain, $k_A(t;\theta)$ can be written explicitly as,
\begin{eqnarray}
k_A(t;\theta) \approx g_m(t;\theta):= \begin{pmatrix}I & 0 & \cdots & 0\end{pmatrix} e^{t G
}\begin{pmatrix}\alpha_1 \\ \alpha_2 \\ \vdots \\ \alpha_m \end{pmatrix}, \quad\quad \text{with } G=\begin{pmatrix}
\;\beta_1\;\; &\; I  \;\;&\; \;\; & \;\; \\ 
\;\beta_2 \;& \;0 \;& \;\ddots \;& \;\; \\
\;\vdots \;&  \;\;& \;\ddots \;&\; I\;   \\
\;\beta_m\; &  \;\;&\; \;&  \; 0\;
\end{pmatrix},\label{eq: ansatz}
\end{eqnarray}
where $I$ denotes the $n$-by-$n$ identity matrix. Here $m$ stands for the order of the rational approximation. In \cite{Ma:16}, the coefficients in the rational approximations were determined based on certain interpolation conditions, and such a rational approximation has been proven to be an excellent approximation for the time correlation function.  For the current problem,  $g_{m}$ will be determined using finitely many essential statistics defined as follows.

\begin{defn}\label{ess_stat}
The values $\{ k_A^{(j)}(t_i;\theta^{\dagger})\}_{i=1,2,\ldots, K}$ are called \textit{essential statistics} if they are sufficient to approximate the response operator $k_A(t;\theta^{\dagger})$ up to order-$m$ using \eqref{eq: ansatz}, for $j =0,\ldots, 2m-1$ and $t_{1}<t_{2}\cdots <t_{K}$, where $j$ indicates the order of derivatives.
\end{defn}

\begin{rem}
In \cite{HLZ:17}, we have considered the derivatives of $k_A$ at $t=0$ for a Langevin dynamics model, and explicit formulas involving the model parameters have been found. Theoretically, all the parameters can be identified from  $\{ k_A^{(j)}(0^+;\theta^{\dagger})\}$, which to some extent shows the consistency.  However, in practice, we use only the derivative of order no greater than one, since estimating higher order derivative of $k_{A}(t,\theta^{\dagger})$ requires time series $X(t_i)$ with sufficient accuracy, which is not necessarily available. By using the essential statistics correspond to lower-order derivatives at $t_i$ other than zero, we give up the explicit expressions of the higher-order statistics with respect to the parameters. This prompted us to solve the problem in the least-squares sense. Furthermore, it requires other more practical methods, such as a sensitivity analysis test, to determine the identifiability of the parameters from these statistics.
\end{rem}

\begin{figure}
\centering
\includegraphics[width=.5\textwidth]{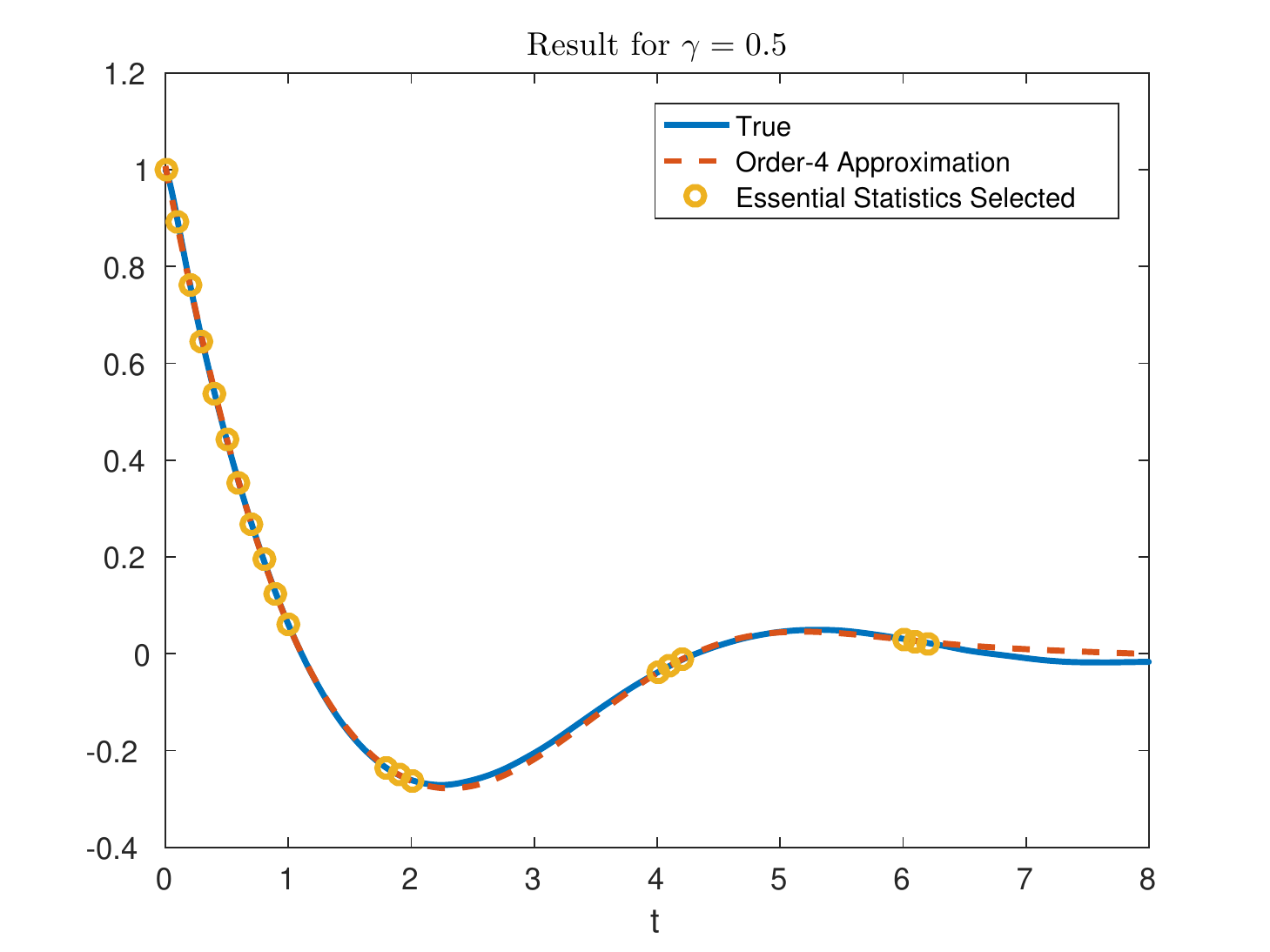}
\caption{The response function (solid lines) and its order-4 approximation (dashed lines).}
\label{eff_rep}
\end{figure}

Figure \ref{eff_rep} shows the performance of such a discrete representation, where the linear response operator arises from a Langevin model subjected to a constant external forcing \eqref{esst_lan}, which will be discussed in Section \ref{sec4}. With such discrete representation we reduce our problem of inferring from an infinite-dimensional $k_{A}(t;\theta^{\dagger})$ to a finite number of essential statistics. To derive a system of equations involving both $\theta$ and those essential statistics we introduce
\begin{eqnarray}
\hat{k}_A(t;\theta):= \mathbb{E}_{{p}_{eq}(\theta)}[A(X(t)) \otimes B^{\dagger}(X(0))].\label{hatkA}
\end{eqnarray}
We also define for $j=0,1,$
\begin{eqnarray}
M_{j}(t_{i}):=k^{(j)}_{A}(t_{i};\theta^{\dagger}) = \frac{\td^{j} }{\td t^{j}}\mathbb{E}_{{p}_{eq}^\dagger}[A(X(t_i)) \otimes B^{\dagger}(X(0))]. \nonumber
\end{eqnarray}
We should stress that \eqref{hatkA} is not the FDT response, since the $B^{\dagger}(x)$ in \eqref{hatkA} is defined with respect to $\peq^{\dagger}$. The key idea in \cite{HLZ:17} is to estimate the true parameter values $\theta^{\dagger}$ by solving
\begin{eqnarray} \label{nonlin_sys}
M_{j}(t_{i})=\hat{k}^{(j)}_A(t_{i};\theta), \quad j\in\{0,1\}, \quad i=1,2,\dots,K.
\end{eqnarray}
Here, the term on the left-hand-side is estimated from the available sample at $p_{eq}^{\dagger}$, and the term on the right-hand-side will be  computed from the time series generated by solving \eqref{true_model} for a given $\theta\in D$.

\begin{rem}\label{remark}
Since $p_{eq}^{\dagger}$ usually depends on a subset of the true parameter value $\theta^\dagger$, the assumption of knowing the explicit form of $p_{eq}^{\dagger}$ as a function of $x$ only is too strong. In general, we cannot even compute the left-hand-side term in \eqref{nonlin_sys} since the term $B^{\dagger}$, which may depend on the unknown true value $\theta^{\dagger}$, is not available to us. We will discuss how to address such issue on two examples in Sections~\ref{sec4} and \ref{sec5}, where each of them belongs to a  class of models with a wide range of applications.
\end{rem}

To ensure the solvability, we assume that the total number of equations in (\ref{nonlin_sys}) is always greater than the dimension of the unknown parameters $\theta$. For this over-determined scenario, we propose to solve (\ref{nonlin_sys}) in the least-squares sense, addressing the practical issue raised in \cite{HLZ:17} {  when the explicit solutions are not available.} The main difficulty in solving the nonlinear least-squares problem comes from evaluating $\hat{k}^{(j)}_A(t_{i};\theta)$, which often requires solving the model (\ref{true_model}) repeatedly. In section \ref{sec3}, we will propose an efficient numerical method to address this issue.

\section{An Efficient Algorithm for Parameter Estimation} \label{sec3}

In section \ref{sec2}, we formulated a parameter estimation method as a problem of solving a nonlinear system (\ref{nonlin_sys}) subject to a dynamical constraint in \eqref{true_model}. In a compact form, we denote the system \eqref{nonlin_sys} as
\begin{equation} \label{nonlin_eqs}
{  f_{i}(\theta):= M_0(t_{i}) - \hat{k}_A(t_{i};\theta) =  0},\quad i=1,\dots,K,
\end{equation}
where $\theta\in [-1,1]^{N}$, $K>N$. { Notice that we have neglected the derivative constraints in \eqref{nonlin_sys}. While similar algorithm and convergence result as shown below also hold if the first-order derivative constraints are included in the least-squares procedure, we will only numerically solve the least-squares problem for the zeroth-order derivative constraints as in \eqref{nonlin_eqs}. As we shall see in Sections~\ref{sec4} and \ref{sec5}, we will use the derivative constraints to directly estimate some of the parameters.}

Our goal is to solve \eqref{nonlin_eqs} in the sense of least-squares, that is,
\begin{equation}\label{nonlin_ls}
\min_{\theta\in [-1,1]^{N}} \sum_{i=1}^{K} f^{2}_{i}(\theta).
\end{equation}
However, as we mentioned in section \ref{sec2}, we do not necessarily have the explicit expressions for $f_{i}(\theta)$ and evaluating $f_{i}$ for certain values of $\theta$ requires solving the true model \eqref{true_model} to approximate the integral in $f_{i}(\theta)$ by a Monte-Carlo sum. For example, if we apply the Gauss-Newton method to solve the nonlinear least-squares problem (\ref{nonlin_ls}), the overall computation will become very expensive since we have to  solve the  model \eqref{true_model} repeatedly. As a remedy, we apply the idea of polynomial based surrogate model, {  which was introduced in \cite{Marzouk:07,Marzouk:09} for solving Bayesian inference problems,} to alleviate this part of computation. {  The main difference of our approach from  \cite{Marzouk:07,Marzouk:09} is that they used the orthogonal polynomials to expand the dynamical model, $X(t,\theta)$, whereas here, we will use the orthogonal polynomials to expand  $f_i(\theta)$. While the convergence of the approximate posterior density estimate of the Bayesian inference problem is established in \cite{Marzouk:09}, we will end this section by analyzing the convergence of the approximate least-squares problem.}

\subsection{Polynomial Based Surrogate Model}

The key idea of our method is to approximate $f_{i}(\theta)$ by a polynomial function $f^{M}_{i}(\theta)$. This is motivated by the fact that the  sampling error cannot be avoided in computing the value of essential statistics, that is, there is uncertainty on the left-hand-side of \eqref{nonlin_sys}. Thus, it is more reasonable to think of the nonlinear least-squares solution of \eqref{nonlin_ls} as a random variable with certain distribution $\mu$ over $[-1,1]^{N}$. Furthermore, since $\{f_{i}(\theta)\}$ are $C^k$ with respect to $\theta$ (Remark~\ref{smooth_dep}) whenever $b$ and $\sigma$ in \eqref{true_model} are $C^k$ with respect to $\theta$, polynomial chaos expansion provides a natural choice of $f^{M}_{i}(\theta)$ based on the orthogonal polynomials with respect to $\mu$. 

For instance, if we consider $\mu\sim U[-1,1]$, i.e., the uniform distribution, then the corresponding orthogonal polynomials are the Legendre polynomials $\{p_{n}\}$. They are given by $p_{0}=1$, $p_{1}=x$ and the recursive formula
\begin{equation}
(n+1)p_{n+1}(x)=(2n+1)xp_{n}(x)-np_{n-1}(x), \nonumber
\end{equation} 
which form a basis for $L^{2}([-1,1],\mu)$. The orthogonality of $\{p_{n}\}$ over $L^{2}([-1,1],\mu)$ is described by
\begin{equation}
\int^{1}_{-1}p_{m}(x)p_{n}(x) \td \mu(x)=\frac{2}{2n+1}\delta_{mn}. \nonumber
\end{equation}
One can normalize $p_n$ by introducing $P_{n}:= p_{n}\sqrt{n+\frac{1}{2}}$.

For multi-dimensional cases, one can define $P_{\vec{k}}(\theta):=\prod_{j=1}^{N}P_{k_j}(\theta_{j})$ ($N$ represents the dimension of $\theta$), where $\vec{k}\geq \vec{0}$ is a multi-index with $\vec{k}=(k_{1},\dots,k_{N})\in\{0,1,2,\dots\}^{N}$ and $\theta=(\theta_{1},\dots,\theta_{N})\in [-1,1]^{N}$. We can consider an order-$M$ approximation of $f_{i}(\theta)$
\begin{equation} \label{expansion}
f_{i}(\theta_{1},\dots,\theta_{N})\approx \sum_{\|\vec{k}\|_{\infty}\leq M} \alpha^{(i)}_{\vec{k}}P_{\vec{k}}(\theta)=:f^{M}_{i}(\theta).
\end{equation}
In practice, there are two approaches to determine the coefficient $\alpha^{(i)}_{\vec{k}}$. The first approach is the Galerkin method, which requires that the errors in the approximation in \eqref{expansion} to be orthogonal to the finite-dimensional subspace of $L^{2}([-1,1],\mu)$ generated by $\{P_{\vec{k}}(\theta)\}_{\|\vec{k}\|_{\infty}\leq M}$. In our convergence analysis, we will deduce the error estimates based on the coefficients obtained through this Galerkin (or least-squares) fitting.
Alternatively, one can determine the coefficient $\alpha_{\vec{k}}^{(i)}$ by the collocation method, that is, matching the values of $f_{i}$ at certain points of $\theta$ in $D$, denoted by $\Theta$, which leads to the following linear equations
\begin{equation}\label{collocation}
f_{i}(\theta)=\sum_{\|\vec{k}\|_{\infty}\leq M} \alpha^{(i)}_{\vec{k}}P_{\vec{k}}(\theta), \quad \theta\in \Theta.
\end{equation}
Since \eqref{collocation} is equivalent to a polynomial interpolation, a common choice of $\Theta$ is the product space of order-$M_C$  Chebyshev nodes, where $M_C\geq M+1$ (when $M_C > M+1$, \eqref{collocation} turns into a linear least-squares problem). In the numerical examples in this paper, we will use collocation method to determine the coefficients.

Replacing $f_{i}(\theta)$ in \eqref{nonlin_ls} by $f_{i}^{M}(\theta)$, we obtain a new least-squares problem for the order-$M$ polynomial based surrogate model,
\begin{equation}\label{poly_ls}
\min_{\theta\in [-1,1]^{N}} \sum_{i=1}^{K} (f^{M}_{i})^{2}(\theta).
\end{equation}
For easier notations, let us use $\textbf{f}^{M}(\theta)$ and $J^{M}(\theta)$ to denote the vector-valued function, $(f^{M}_{1}(\theta),\dots, f^{M}_{N}(\theta))^{\top}$, and its Jacobian matrix, respectively. 

In summary, we have the following algorithm \ref{alg:surrogate} based on the Gauss-Newton method.

\begin{algorithm}
\caption{Parameter Estimation: Polynomial Based Surrogate Model}
\label{alg:surrogate}
\begin{algorithmic}
\STATE{Let $\Theta$ be a set of collocation nodes in $[-1,1]^{N}$.}
\STATE{1. For each $\theta \in \Theta$, solve \eqref{true_model}.}
\STATE{2. Estimate $f_{i}(\theta)$ using a Monte-Carlo sum over $i=1,\dots,K$.}
\STATE{3. Compute the coefficients $\alpha^{(i)}_{\vec{k}}$ by solving the linear system \eqref{collocation} and obtain the approximations $f^{M}_{i}(\theta)$ for $i=1,\dots,K.$}
\STATE{4. For an initial guess $\theta^{0}\in [-1,1]^{N}$ and a threshold $\delta>0$. Compute }
\STATE{
	\begin{equation}\label{G-N}
	\theta^{k}=\theta^{k-1}-[(J^{M}(\theta^{k-1}))^{\top}J^{M}(\theta^{k-1})]^{-1}(J^{M}(\theta^{k-1}))^{\top}\textbf{f}^{M}(\theta^{k-1}).
	\end{equation}
}
\WHILE{The step length $\|\theta^{k}-\theta^{k-1}\|\geq \delta$}
\STATE{Repeat \eqref{G-N}.}
\ENDWHILE\\
\RETURN {$\theta^{k}$}
\end{algorithmic}
\end{algorithm}

\comment{

\begin{algorithm}[H]
\caption{Parameter Estimation: Polynomial Based Surrogate Model}
\label{alg:surrogate}
\begin{algorithmic}
\State{For the Galerkin method, we introduce a sparse grid over $[-1,1]^{N}$ denoted by $G$(or a set of nodes over $[-1,1]^{N}$ denoted by $\Theta$ for the collocation method).}
\State{1. For each $\theta\in G$ (or $\theta \in \Theta$), solve \eqref{true_model}.}
\State{2. Estimate $f_{i}(\theta)$ using a Monte-Carlo sum over $i=1,\dots,K$.}
\State{3. Compute the coefficients $\alpha^{(i)}_{\vec{k}}$ by (\ref{coef}) using a sparse grid quadrature scheme (or by solving the linear system \eqref{collocation}), and obtain the approximations $f^{M}_{i}(\theta)$ for $i=1,\dots,K.$}
\State{4. For an initial guess $\theta^{0}\in [-1,1]^{N}$ and a threshold $\delta>0$. Compute }
\State{
	\begin{equation}\label{G-N}
	\theta^{k}=\theta^{k-1}-[(J^{M}(\theta^{k-1}))^{\top}J^{M}(\theta^{k-1})]^{-1}(J^{M}(\theta^{k-1}))^{\top}\textbf{f}^{M}(\theta^{k-1}).
	\end{equation}
}
\While{The step length $\|\theta^{k}-\theta^{k-1}\|\geq \delta$}
\State{Repeat \eqref{G-N}.}
\EndWhile\\
\Return {$\theta^{k}$}
\end{algorithmic}
\end{algorithm}
}

Computationally, the most expensive step is in solving \eqref{true_model} on the collocation nodes $\Theta$, which can be done in parallel.

\subsection{The Convergence of the Approximate Solutions}\label{subsec:con} 

In this section, we will provide the convergence analysis of the solutions of the approximate least-squares problem in   \eqref{poly_ls} to the solution of the true least-squares problem in \eqref{nonlin_ls} as $M \to \infty$, under appropriate conditions. 
We denote the solution of \eqref{poly_ls} by $\theta^{*}_{M}$, that is,
\begin{equation*}
\theta^{*}_{M}:= \arg\min_{\theta\in [-1,1]^{N}} \sum_{i=1}^{K} (f^{M}_{i})^{2}(\theta).
\end{equation*}
Since each $f_{i}^{M}$ is a multivariate polynomial, the minimum \eqref{poly_ls} can be attained in $[-1,1]^{N}$; but such $\theta^{*}_{M}$ may not be unique. For the case of $K=N$ and sufficiently smooth $f_i$, assuming that the true parameter value $\theta^{\dagger}\in (-1,1)^{N}$ is the unique solution of the true least-squares problem in \eqref{nonlin_ls}, we can show that there exists a sequence of minimizers of the approximated problem \eqref{poly_ls} $\{\theta^{*}_{M}\}$ such that $\theta^{*}_{M}\rightarrow \theta^{\dagger}$ as $M\rightarrow+\infty$. To be more precise, 
 
\comment{Before we apply numerical method to solve the problem, a fundamental question will be: assuming $\theta^{\dagger}\in (-1,1)^{N}$, the true parameter value, is the unique solution of \eqref{nonlin_ls} can we find a sequence of minimizers of the approximated problem \eqref{poly_ls} $\{\theta^{*}_{M}\}$ such that $\theta^{*}_{M}\rightarrow \theta^{\dagger}$ as $M\rightarrow+\infty$? We are able to prove the following theorem, which partially answers this question.}

\begin{thm}\label{thm:convN}
	Consider the $N$-dimensional nonlinear least-squares problem \eqref{nonlin_ls} and its surrogate model \eqref{poly_ls}. In particular, we let $K=N$, and assume
	\begin{enumerate}
		\item $f_{i}\in C^{\ell}([-1,1]^{N})$, $\ell=\left\lceil \frac{3}{2}N\right\rceil+2$, $i=1,2,\dots,N$, 
		\item Let $\theta^\dagger$ be the solution of \eqref{nonlin_ls} such that $f_{i}(\theta^{\dagger})=0$, $i=1,2,\dots,N$,
		\item the Jacobian matrix of $\textbf{f}:=(f_1,f_2,\dots, f_{N})^{\top}$ at $\theta^{\dagger}$, denoted  by $\nabla \vec{f}(\theta^{\dagger})$,  is invertible.
	\end{enumerate}
    Then there exists a sequence of minimizers $\{\theta^{*}_{M}\}$ such that
    \begin{equation*}
\theta^{*}_{M}   = \text{argmin}_{\theta\in [-1,1]^{N}} \sum_{i=1}^{N} (f^{M}_{i})^{2}(\theta)
    \end{equation*}
    and $\theta^{*}_{M}\rightarrow \theta^{\dagger}$ as $M\rightarrow +\infty$. Moreover, the residual error $\textbf{f}^{M}(\theta^{*}_{M})$ converges to $0$ as $M\rightarrow +\infty$.
\end{thm}
\begin{rem}
{ We set $N=K$ since the existence of the minimizers rely on the technical assumption in Lemma~\ref{lem:p4} below.}  The second and third assumptions provide the well-posedness of the original nonlinear least-squares problem \eqref{nonlin_ls}. Since we discuss the parameter estimation problem under the perfect model setting, the second assumption naturally holds.
\end{rem}
	The proof of theorem \ref{thm:convN} can be found in the Appendix~\ref{appA}. Here we will present the proof for the one-dimensional case ($K=N=1$) to illustrate the main ideas. Notice that in the one-dimensional case, the formula in \eqref{expansion} is reduced to
\begin{equation}\label{eq:oned}
f^{M}(\theta):= \sum_{k=0}^{M} \alpha_{k} P_{k}(\theta), 
\end{equation}
where $\{P_{n}\}$ are normalized Legendre polynomials. We start with two Lemmas, which pertain to the pointwise convergence of $f^{M}$, and the proof can be found in \cite{isaacson2012analysis}.
\begin{lem}\label{lem:p}
Let $f\in C^{\ell}([-1,1])$, then the corresponding least-squares approximation given by \eqref{eq:oned} satisfies
	\begin{equation*}
	\lim_{k\rightarrow +\infty} \alpha_{k}k^{\ell}=0.
	\end{equation*}
\end{lem}
Lemma \ref{lem:p} connects the smoothness  of $f$ and  the decay rate of the generalized Fourier coefficients $\alpha_{k}$. The following Lemma is a direct result of Lemma \ref{lem:p}.
\begin{lem}\label{lem:p1}
	Let $f\in C^{2}([-1,1])$ and $f^M$ be the corresponding least-squares approximation given by \eqref{eq:oned}, then for any $\epsilon>0$, there exists $N>0$ such that $\forall M>N$ we have
	\begin{equation*}
	\|f-f^{M}\|_{\infty}\leq \frac{\epsilon}{\sqrt{M}}.
	\end{equation*}
	Thus, we have $f^{M}\rightarrow f$ pointwise on $[-1,1]$.
\end{lem}
Here $\|\cdot\|_{\infty}$ denotes the $L^{\infty}$-norm on $C^2[-1,1]$. In the proof of proposition \ref{thm:p3}, we also need $\|(f^{M})'-f'\|_{\infty} \to 0$ as $M\rightarrow +\infty$. To obtain such convergence, one needs higher regularity on $f$. In particular, we have:
\begin{lem}\label{thm:p2}
	Let $f\in C^{4}([-1,1])$ and $f^M$ be the corresponding least-squares approximation given by \eqref{eq:oned}. We have $\|f'-(f^{M})'\|_{\infty}\rightarrow 0$ as  $M\rightarrow +\infty$.
\end{lem}
\begin{proof}
	Following the same idea in the proof of Lemma \ref{lem:p1} \cite{isaacson2012analysis}, we are going to show that  $\{(f^{M})'\}$ is a Cauchy sequence under $L^{\infty}$-norm. Notice for $m>n$,
	\begin{equation*}
	\|(f^{m}-f^{n})'\| _{\infty} =\left \|\sum_{k=n+1}^{m} \alpha_{k}P'_{k}\right\|_{\infty}
	\leq \sum_{k=n+1}^{m} |\alpha_{k}|\|P'_{k}\|_{\infty},
	\end{equation*}
	which suggests that we need to find a $L^{\infty}$ bound for $P'_{k}$. Using the recursion relation $p'_{n+1}-p_{n-1}=(2n+1)p_{n}$ and $\|p_{n}\|_{\infty}=1$, we have 
	\begin{equation}\label{eq:pbound}
	\|P'_{n}\|_{\infty}=\sqrt{n+\frac{1}{2}}\|p'_{n}\|_{\infty}\leq \frac{1}{2}n(n+1)\sqrt{n+\frac{1}{2}}.
	\end{equation}
	As a result,
	\begin{equation*}
	\|(f^{m}-f^{n})'\|_{\infty}\leq \frac{1}{2}\sum_{k=n+1}^{m}|\alpha_{k}| k(k+1)\sqrt{k+\frac{1}{2}}.
	\end{equation*}
	From Lemma \ref{lem:p} we know that $f\in C^{4}([-1,1])$ is enough to ensure $\{(f^{M})'\}$ is Cauchy. Let $(f^{M})'\rightarrow \varphi' \in C^{0}([-1,1])$ in $L^{\infty}$ sense. The remaining issue is whether $\varphi'=f'$. This is straightforward, since by Lemma \ref{lem:p1} we know $f^{M}\rightarrow f$ pointwisely. Moreover,
	\begin{equation*}
	\begin{split}
	f^{M}(\theta)-f^{M}(-1) &=\int^{\theta}_{-1}(f^{M})'(t)\td t, \\
	\lim_{M\rightarrow +\infty} \int^{\theta}_{-1}(f^{M})'(t)\td t &= \int^{\theta}_{-1} \varphi'(t)\td t, \quad  \forall \theta\in[-1,1],
	\end{split}
	\end{equation*}
	and the convergence is uniform with respect to $x$. Thus, $\forall x\in [-1,1],$
	\begin{equation*}
	\lim_{M \rightarrow +\infty}f^{M}(\theta)-f^{M}(-1)=f(\theta)-f(-1) =\int_{-1}^{\theta}\varphi'(t)\td t,
	\end{equation*}
	that is, $f'=\varphi'.$ 
\end{proof}
So far, we have shown that with enough regularity of $f$, we have the uniform convergence of both $f^{M}\rightarrow f$ and $(f^{M})'\rightarrow f'$ on $[-1,1]$. To show the existence of minimizer near $\theta^{\dagger}$, the following Lemma is crucial.

\begin{lem}\label{lem:p4}
	Let $F: D\subset \mathbb{R}^{n}\rightarrow \mathbb{R}$ be continuously differentiable in the domain $D$ and suppose that there is an open ball $B(x^{0},r)\subset D$ and a positive $\gamma$ such that $\|\nabla F(x)^{-1}\|\leq \gamma$ $\forall x\in B(x^{0},r)$ and $r> \gamma \|F(x^{0})\|$. Then $F(x)=0$ has a solution in $B(x^{0},r)$.
\end{lem}
The proof can be found in \cite{ortega1970iterative}. In the one-dimensional case, Lemma \ref{lem:p4} basically says that if a $C^{1}$ function $f$ changes with a minimum speed ($|f'|>\gamma^{-1}$) on the interval $(x^0-r,x^0+r)$ and the function value at $x^0$ is small enough ($f(x^0)\leq \frac{r}{\gamma}$), then the function must reach zero in this interval. With Lemma \ref{lem:p1}-\ref{lem:p4}, we are able to prove the following result which is the one-dimensional analogy of theorem \ref{thm:convN}.
\begin{prop}\label{thm:p3}
		Consider the following one-dimensional nonlinear least-squares problem
		\begin{equation*}
		\min_{\theta\in [-1,1]} f^{2}(\theta),
		\end{equation*}
		with solution $\theta^{\dagger}$. If $f$ is approximated by $f^{M}$ given by \eqref{eq:oned} and we assume that
		\begin{enumerate}
			\item $f\in C^{4}([-1,1])$,
			\item $f(\theta^{\dagger})=0$,
			\item $|f'(\theta^{\dagger})|\not=0$,  
		\end{enumerate}
		then there exists a sequence of minimizer $\{\theta^{*}_{M}\}$ such that
		\begin{equation*}
		(f^{M})^{2}(\theta^{*}_{M})= \min_{\theta\in [-1,1]} (f^{M})^{2}(\theta)=: e^2_{M}
		\end{equation*}
		and 
		\begin{equation*}
		\lim_{M\rightarrow +\infty}\theta^{*}_{M}= \theta^{\dagger}, \quad \lim_{M \rightarrow +\infty}|f^{M}(\theta^{*}_{M})|=0.
		\end{equation*} 
\end{prop}
\begin{proof}
	From Lemma \ref{lem:p1} and \ref{thm:p2}, we know that the first assumption provides
\begin{equation*}
\lim_{n \rightarrow +\infty} \|f-f^{M}\|_{\infty}=0, \quad \lim_{n \rightarrow +\infty} \|f'-(f^{M})'\|_{\infty}=0.
\end{equation*}
Let $F^{M}(\theta):=f^{M}(\theta)-e_{M}$, and we are going to apply Lemma \ref{lem:p4} to $F^{M}$ at $\theta^{\dagger}$. Since $f'(\theta^{\dagger})\not =0$, by continuity we know that there exist positive constants $\gamma$ and $r$ such that
\begin{equation}\label{eq:psugo}
|f'(\theta)|^{-1}\leq \frac{\gamma}{2}, \quad \forall \theta \in (\theta^{\dagger}-r,\theta^{\dagger}+r)\subset (-1,1).
\end{equation}
Further notice that $(F^{M})'=(f^{M})'$ and $\|f'-(f^{M})'\|_{\infty}\rightarrow 0$. This implies that there exists a positive constant $L_{1}$ such that $\forall M>L_1$
\begin{equation}\label{eq:cond1}
|(F^{M})'(\theta)^{-1}| \leq \gamma, \quad \forall \theta \in (\theta^{\dagger}-r, \theta^{\dagger}+r)\subset (-1,1).
\end{equation}
At $\theta=\theta^{\dagger}$, we have 
\begin{equation*}
|F^{M}(\theta^{\dagger})|  =|f^{M}(\theta^{\dagger})-e_{M}|\leq  |f^{M}(\theta^{\dagger})|+e_{M} \leq 2|f^{M}(\theta^{\dagger})|,
\end{equation*}
where we used the fact that $e_{M} \leq |f^{M}(\theta^{\dagger})| $, since $e^2_{M}$ is the minimum of $(f^{M})^2$ on $[-1,1]$. As a result,
\begin{equation*}
\lim_{M\rightarrow +\infty}|F^{M}(\theta^{\dagger})| \leq \lim_{M\rightarrow +\infty} 2|f^{M}(\theta^{\dagger})|=2 |f(\theta^{\dagger})|=0.
\end{equation*}
Thus, we are able to select $L_2>0$ such that $\forall M>L_2$
\begin{equation}\label{eq:cond2}
|F^{M}(\theta^{\dagger})|< \frac{r}{\gamma}.
\end{equation}
Since we have both \eqref{eq:cond1} and \eqref{eq:cond2} $\forall M> \max\{L_1,L_2\}=:L$, applying Lemma \ref{lem:p4} to $F^{M}$ at $\theta^{\dagger}$, we conclude that $F^{M}(\theta)=0$ has a solution in $(\theta^{\dagger}-r,\theta^{\dagger}+r)$. Denote such solution as $\theta^{*}_{M}$, which is a minimizer of the approximated least-squares problem, that is,
\begin{equation*}
(f^{M})^{2}(\theta^{*}_{M}) =e^2_{M}= \min_{\theta\in [-1,1]} (f^{M})^{2} (\theta), \quad \forall M>L.
\end{equation*}
Apply the mean value theorem to $f(\theta^{*}_{M})-f(\theta^{\dagger})$ for $M>L$ we have
\begin{equation*}
|\theta^{\dagger}-\theta^{*}_{M}| \leq |f'(\xi_{M})|^{-1} |f(\theta^{\dagger})-f(\theta^{*}_{M}))|=|f'(\xi_{M})|^{-1} |f(\theta^{*}_{M})|,
\end{equation*}
for some $\xi_{M}\in (\theta^{\dagger}-r, \theta^{\dagger}+r)$. Thus, the stability condition in \eqref{eq:psugo} leads to
\begin{equation*}
\begin{split}
 |\theta^{\dagger}-\theta^{*}_{M}| &   \leq \frac{\gamma}{2} |f(\theta^{*}_{M})| \leq \frac{\gamma}{2} \left( |f(\theta^{*}_{M})-f^{M}(\theta^{*}_M)|+|f^{M}(\theta^{*}_{M})|\right) \\
& \leq  \frac{\gamma}{2} \left(\|f-f^{M}\|_{\infty}+|f^{M}(\theta^{\dagger})|\right) \leq \frac{\gamma \epsilon}{\sqrt{M}}  \rightarrow 0.
\end{split}
\end{equation*}
as $M\to\infty$ and for any $\epsilon>0$. Here, we have used Lemma~\ref{lem:p1} to bound the first term and the residual error, 
\begin{equation*}
|f^{M}(\theta^{*}_{M})|\leq |f^{M}(\theta^{\dagger})|=|f^{M}(\theta^{\dagger}) - f(\theta^{\dagger})|\leq\frac{\epsilon}{\sqrt{M}}.
\end{equation*}
\end{proof}
This result guarantees the existence of minimizers of the approximate polynomial least-squares problem that converges to the solution of the true least-squares problem under a fairly strong condition, that is, $f$ is smooth enough. In practice (see Section~\ref{sec4}), we shall see reasonably accurate estimates even when the essential statistics are not as regular as required by this theoretical estimate.

\subsection{Convergence of the Algorithm}

In algorithm \ref{alg:surrogate}, we are solving a least-squares problem (\ref{poly_ls}). For the Gauss-Newton method, a standard local convergence result under a full rank assumption of the Jacobian matrix is well known (e.g., see Chapter 2 of \cite{Kelley} for details). To achieve such a local convergence in our numerical algorithm, we will now verify the full rank condition in the polynomial chaos setting.

\comment{The analysis in \cite{Haussler:86} provides a Kantorovich-type of convergence result for the Gauss-Newton method without the full rank requirement, and similarly in \cite{Smale:86,Dedieu:00}, a Smale-type of convergence analysis for the Gauss-Newton method was carried out. We are going to consider the result in \cite{Kelley}, and verify the full rank condition in the polynomial chaos setting.

\begin{thm}\label{local_convergence}
	Let $\theta^{*}$ be the solution of the least-squares problem \eqref{poly_ls}. If the Jacobian matrix $J^{M}(\theta^{*})$ has full column rank, then there exist $K>0$ and $\delta>0$ such that if $\theta^{0}\in \mathcal{O}(\theta^{*},\delta)$ then the error in the Gauss-Newton iteration satisfies
	\begin{equation}\label{eq:convergence}
	\|e^{1}\|\leq K(\|e^{0}\|^{2}+\|\textbf{f}^{M}(\theta^{*})\|\|e^{0}\|),
	\end{equation}
	where the error $e^{i}:=\theta^{*}-\theta^{i}$ for $i=0,1$ and $\theta^{1}$ is defined by \eqref{G-N}. Here constant $K$ only depends on $\delta$, $\theta^{*}$, and $\textbf{f}^{M}$.
	\end{thm}
	
	Notice that for zero-residual problem, that is, ${\textbf{f}}^{M}(\theta^{*})=0$, the $\|\textbf{f}^{M}(\theta^{*})\|\|e^{0}\|$ term on the right-hand side of \eqref{eq:convergence} vanishes, which leads to
	\begin{equation*}
	\frac{\|e^{1}\|}{\|e^{0}\|}\leq K\|e^{0}\|< K\delta.
	\end{equation*}
	Hence, in our case, to achieve a fast local convergence of Gauss-Newton method we need the residual $\|\textbf{f}^{M}(\theta^{*})\|$ to be small.
}

For our problem, the Jacobian matrix $J^{M}(\theta)\in \bb{R}^{K\times N}$ is given by
\begin{equation}
(J^{M})_{ij}=\frac{\partial f^{M}_{i}}{\partial \theta_{j}}, \quad i=1,\dots,K, \quad  j=1,\dots N, \nonumber
\end{equation}
where $f^{M}_{i}$ is an element of the finite-dimensional polynomial space $\Gamma^{M}_{N}$ defined by
\begin{equation*}
\Gamma^{M}_{N}:=\spn\{\theta_{1}^{k_{1}}\theta_{2}^{k_{2}}\cdots \theta_{N}^{k_{N}}\;|\; k_{i}=0,1,\dots,M\}.
\end{equation*}
Using linear independence over the function space $\Gamma^{M}_{N}$ and some fundamental results in algebra, we can prove the following theorem.
\begin{thm}\label{full_rank}
If $F:=\{f_{i}\}_{i=1}^{K}\subset \Gamma^{M}_{N}$ is linearly independent and $K> \max\{N,(M+1)^{N-1}\}$, then the column space of the corresponding Jacobian matrix $J(\theta)$ is full rank in $(\Gamma^{M}_{N})^{K}$. Furthermore, the set $\mathcal{N}(J)$ defined by
\begin{equation*}
\mathcal{N}(J):=\{\theta\in\bb{R}^{N}\;|\; \rank(J(\theta))\leq N-1\},
\end{equation*}
is nowhere dense over $\bb{R}^{N}$.
\end{thm}
\begin{proof}
See Appendix~\ref{appB}.
\end{proof}
\begin{rem}
Here $\mathcal{N}(J)$ is the set of $\;\theta$ for which $J(\theta)$ is not full rank in the matrix sense. In our application, $N$ is the dimension of the parameter $\theta$, $K$ is the total number of equations in (\ref{nonlin_eqs}) and $M$ is the order of approximation used in (\ref{expansion}). In practice, since the orthogonal polynomials involved in \eqref{expansion} form a basis for $\Gamma^{M}_{N}$, it is enough to check the rank of the coefficient matrix $A^{M}$ where the $i^{\text{th}}$ column consists of the coefficients $\alpha^{(i)}_{\vec{k}}$ in \eqref{expansion} for all possible $\vec{k}$. The linear independence hypothesis is indeed sensible in practice since we want to avoid solving underdetermined least-squares problems.
\end{rem}

\section{Example I: The Langevin Model} \label{sec4}

For the first example, we consider a classical model in statistical mechanics: the dynamics of a particle driven by a conservative force, a damping force, and a stochastic force. In particular, we choose the conservative force based on the Morse potential
\begin{equation}
U(x)= U_0(a(x-x_0)), \quad U_0(x)= \epsilon(e^{-2x}-2e^{-x} + 0.01 x^2), \label{potential}
\end{equation}
where the last quadratic term in $U_{0}$ acts as a retaining potential (also known as a confining potential), preventing the particle from moving to infinity. For this one-dimensional model, we rescale the mass to unity $m=1$, and write the dynamics as follows
\begin{equation}\label{Lan_sys}
\begin{cases}
\dot{x}=v  \\
\dot{v}=-U'(x)-\gamma v+\sqrt{2\gamma k_{B}T} \dot{W}, \\
\end{cases}
\end{equation}
where $\dot{W}$ is a white noise. The generator of the system \eqref{Lan_sys}, denoted by $\OL$, is given by
\begin{equation}\label{generator}
\OL= v\frac{\partial}{\partial x}+(-U'(x)-\gamma v)\frac{\partial}{\partial v}+\gamma k_{B}T \frac{\partial^{2}}{\partial v^2}. 
\end{equation}
The smooth retaining potential $U(x)$ guarantees the ergodicity of the Langevin system \eqref{Lan_sys} {  (see Appendix~\ref{appC} for details)}. Namely, there is an equilibrium distribution (Gibbs
measure) $p_{eq}(x,v)$, given by
\begin{equation}\label{eqdis} 
p_{eq}\propto \exp\Big[-\frac{1}{k_{B}T}\big(U(x)+\frac{1}{2}v^{2}\big)\Big],
\end{equation}
which is independent of $\gamma$. In particular, we have $v\sim \mathcal{N}(0,k_{B}T)$ at equilibrium. For this Langevin model, there are a total of five unknown parameters $\theta:=(\gamma,T,\epsilon,a,x_0)$ with true values $\theta^{\dagger}:=(\gamma^{\dagger},T^{\dagger},\epsilon^{\dagger},a^{\dagger},x^{\dagger}_0)$. 

{  For this specific problem, notice that all parameters except $\gamma$ appear in the equilibrium density function. Intuitively, this suggests that one can estimate four parameters, $\{T,\epsilon,a,x_0\}$, from equilibrium statistics and $\gamma$ from two-point statistics, respectively. We will present a conventional estimation method based on this guideline for diagnostic purpose. In particular, the parameter $\gamma$ will be directly estimated from a two-point statistic as shown in Section~\ref{secA} below. Subsequently, the remaining four parameters are obtained via one-point statistics, as shown in \eqref{out_tg} below and Section \ref{secB1}. 

As for the proposed approach, we will consider estimating all the parameters using the essential statistics in the most efficient manner. First, two of the parameters, $\{\gamma,T\}$, will be estimated using the essential statistics in \eqref{out_tg}. Then, via a sensitivity analysis discussed in Section \ref{local_sens_analysis}, we will show that parameter $x_0$ is independent of the essential statistics. As a consequence, we formulate a least-square problem to estimate $\theta=(\epsilon, a)$ and use equilibrium statistics to estimate $x_0$.}

\subsection{ Reduction of the Parameter Space} \label{secA}
In our previous work \cite{HLZ:17}, we have suggested an approach to estimate $T^{\dagger}$ and $\gamma^{\dagger}$ directly from the essential statistics. To construct such essential statistics, we consider a constant external forcing $\delta f$ with $\delta\ll 1$. The corresponding perturbed system is given by
\begin{equation}\label{psys}
\begin{cases}
\dot{x}=v  \\
\dot{v}=-U'(x)-\gamma v+\delta f+\sqrt{2\gamma k_BT} \dot{W}, \\
\end{cases}
\end{equation}
that is, $c(x)=(0,1)$ in the FDT formula \eqref{RA}. By selecting the observable $A=(0,v)^\top$, we work with the $(2,2)$ entry of the response operator given by
\begin{equation}\label{esst_lan}
k_{A}(t;\theta^{\dagger})=\frac{1}{k_B T^{\dagger}}\mathbb{E}_{\peq^{\dagger}}[v(t)v(0)]=\frac{1}{k_B T^{\dagger}}\int_{\mathbb{R}^2}\int_{\mathbb{R}^2} v v_0\rho(x,v,t|x_0,v_0,0)p^{\dagger}_{eq}(x_{0},v_{0}) \td x\td v\td x_0 \td v_0,
\end{equation}
where $\rho$ is the solution of corresponding Fokker-Planck equation \eqref{fpe} associated with the Langevin dynamics. Taking the time derivative of $\mathbb{E}_{\peq^{\dagger}}[v(t)v(0)]$ for $t>0$, we obtain
\begin{equation}\label{diff_with_gen}
\begin{split}
\frac{\partial}{\partial t}\mathbb{E}_{\peq^{\dagger}}[v(t)v(0)] &=\int_{\mathbb{R}^2}\int_{\mathbb{R}^2} 
v v_0 \OL^{*}\rho(x,v,t|x_0,v_0,0)p^{\dagger}_{eq}(x_{0},v_{0}) \td x\td v\td x_0 \td v_0 \\
&=\int_{\mathbb{R}^2}\int_{\mathbb{R}^2} 
\OL v\; v_0 \rho(x,v,t|x_0,v_0,0)p^{\dagger}_{eq}(x_{0},v_{0}) \td x\td v\td x_0 \td v_0 \\ 
&=\int_{\mathbb{R}^2}\int_{\mathbb{R}^2} 
(-U'(x)-\gamma v) v_0 \rho(x,v,t|x_0,v_0,0)p^{\dagger}_{eq}(x_{0},v_{0}) \td x\td v\td x_0 \td v_0\\
&=\mathbb{E}_{\peq^{\dagger}}[(-U'(x(t))-\gamma^{\dagger} v(t))v(0)].\\
\end{split} 
\end{equation}
Let $t\rightarrow 0^{+}$ in \eqref{diff_with_gen} and recall that $v\sim\mathcal{N}(0,k_{B}T)$ at equilibrium, we have
\begin{equation} \label{out_tg}
\begin{split}
\mathbb{E}_{p^{\dagger}_{eq}}[v(t)v(0)]\Big|_{t=0} &=k_{B}T^{\dagger}; \\ \frac{\partial}{\partial t}\mathbb{E}_{p^{\dagger}_{eq}} [v(t)v(0)]\Big|_{t=0^+} &= -\gamma^{\dagger} k_{B}T^{\dagger},
\end{split}
\end{equation}
where both the left-hand-sides can be computed from the sample. Thus, \eqref{out_tg} provides direct estimates for $T^{\dagger}$ and $\gamma^{\dagger}$. As a result, the original parameter estimation problem can be reduced into estimating $(\epsilon^{\dagger}, a^{\dagger}, x^{\dagger}_0)$ in the potential function $U$, which is a non-trivial task since the dependence of $p_{eq}$ on these parameters is quite complicated. With these estimates, $B^\dagger$ becomes available since it only depends on $k_BT^\dagger$, which allows one to compute $M_j(t_i)$ in \eqref{nonlin_sys} and avoid the issue pointed out in Remark~\ref{remark}.

\subsection{Parameter Estimation Approaches}

In this subsection, we focus on the three-dimensional parameter estimation problem with $\theta:=(\epsilon,a,x_{0})$. First, we review the conventional approach which matches the equilibrium statistics of $x$. Then, we discuss the proposed new approach using the essential statistics. In particular, we perform a sensitivity analysis to determine the identifiability of the parameters from the essential statistics. Finally, we present the numerical results including a comparison between the conventional and the new approaches. 

\subsubsection{Conventional Method}\label{secB1}
We first look at what we can learn from the equilibrium statistics. Since the marginal distribution of $x$ at equilibrium state is proportional to $\exp(-U(x;\theta)/k_{B}T)$, a natural idea is to match the moments of $x$ with respect to this equilibrium density. In particular, one can introduce the following three equations
\begin{eqnarray}\label{match_eq}
\mathbb{E}_{p^{\dagger}_{eq}}[x^{i}]=\mathbb{E}_{p_{eq}}[x^{i};\epsilon,a,x_0], \quad i=1,2,3,
\end{eqnarray}
where the left-hand-sides are computed from the given data while the right-hand-sides are treated as functions of $(\epsilon, a, x_0)$, obtained from solving the model \eqref{psys}. The following pre-computation simplifies equation \eqref{match_eq} into a one-dimensional problem.

To begin with, we define the probability density functions
\begin{equation}
\begin{split}
p_{eq}^{a, x_0}(x)&:=\frac{1}{N}\exp\big(-\frac{1}{k_B T}U_0(a(x-x_0))\big);\\
p_{eq}^{1,0}(x)&:=\frac{1}{N_{0}}\exp\big(-\frac{1}{k_B T}U_0(x)\big),\nonumber
\end{split}
\end{equation}
in terms of $U$ and $U_0$, respectively. With a change of variables $y:=a(x-x_0)$ the normalizing constants $N$ and $N_0$ satisfy
\begin{equation*}
\begin{split}
N=\int_{\mathbb{R}}\exp\big(-\frac{1}{k_B T}U_0(a(x-x_0))\big)\td x \\ =\frac{1}{a}\int_{\mathbb{R}}\exp(-\frac{1}{k_B T}U_0(y))\td y=\frac{1}{a}N_0.
\end{split}
\end{equation*}
As a result, the first two equations in \eqref{match_eq} can be written into
\begin{eqnarray}
\mathbb{E}_{p_{eq}^{\dagger}}[x] &=& \mathbb{E}_{p_{eq}}^{a,x_0}[x]=\frac{1}{a}\mathbb{E}_{p_{eq}}^{1,0}[x]+x_0,\nonumber\\
\mathbb{E}_{p_{eq}^{\dagger}}[x^2] &=& \mathbb{E}_{p_{eq}}^{a,x_0}[x^2] =\frac{1}{a^2}\mathbb{E}_{p_{eq}}^{1,0}[x^2]+\frac{2 x_0}{a}\mathbb{E}^{1,0}_{p_{eq}}[x]+x_0^2,\nonumber
\end{eqnarray}
with a unique solution
\begin{equation}\label{axo}
a=\sqrt{\frac{\Var^{1,0}_{p_{eq}}[x]}{\Var_{\peq^{\dagger}}[x]}},\quad
x_0=\mathbb{E}_{\peq^{\dagger}}[x]-\mathbb{E}^{1,0}_{p_{eq}}[x]\sqrt{\frac{\Var_{\peq^{\dagger}}[x]}{\Var^{1,0}_{p_{eq}}[x]}},
\end{equation}
where $\Var_{p_{eq}}^{1,0}[x]$ stands for the variance of $x$ with respect to the rescaled equilibrium density $p^{1,0}_{eq}$. In practice, both $\mathbb{E}_{\peq^{\dagger}}[x]$ and $\mbox{Var}_{\peq^{\dagger}}[x]$ can be empirically estimated directly from the available data, while $\Var_{p_{eq}}^{1,0}[x]$ is a function of $\epsilon$ only. Thus, (\ref{axo}) can be denoted as
\begin{equation}
a=a(\epsilon);\quad x_{0}=x_{0}(\epsilon),\nonumber
\end{equation}
and \eqref{match_eq} is reduced to a one-dimension problem for $\epsilon$,
\begin{equation} \label{match_eq1}
\begin{split}
\mathbb{E}_{p^{\dagger}_{eq}}[x^3] &= \mathbb{E}_{p_{eq}}[x^3;\epsilon,a,x_0] {  \left(= \int_{\mathbb{R}} x^3 \peq(x;\epsilon,a,x_0)\td x \right)} \\
\text{ s.t. } a &=a(\epsilon), \; x_0=x_0(\epsilon).
\end{split}
\end{equation}
{ Note that the left-hand-side of \eqref{match_eq1} will be estimated by Monte-Carlo averaging from the unperturbed data, whereas the right-hand-side integral will be approximated using a numerical quadrature.} 

\subsubsection{Sensitivity Analysis of the Essential Statistics}\label{local_sens_analysis}

To ensure a successful parameter estimation, it is useful to gather some a priori knowledge of the parameter identifiability from the proposed essential statistics. This is important especially since there are non-unique choices of essential statistics to be fit and we want to ensure that we can identify the parameters from appropriate essential statistics: Recall that the essential statistics are defined based on the choices of external forcing and observable. 

While local sensitivity analysis such as the pathwise derivative method described in \cite{srk:12} is desirable, it requires knowledge of the true parameters which are not available to us. Essentially, the pathwise derivative method is to compute $\bar{y}_\theta := \mathbb{E} [Y(t,\omega)]$, where $Y(t,\omega):=\mathcal{D}_{\theta}X(t,\theta)$ and the expectation is defined with respect to the density of the It\^o diffusion (written in its integral form),
\begin{equation} 
X(t,\theta)=X_{0}(\theta)+\int^{t}_{0}b(X(s,\theta),\theta) \td s+\int^{t}_{0}\sigma(X(s,\theta),\theta)\td W_{s}.\label{Xt}
\end{equation}
Since we have no explicit solutions of the density $\rho(x,t;\theta)$, one way to approximate $\bar{y}_\theta$ is using an ensemble average of the solutions of the following equation, 
\begin{equation}\label{Yt}
\begin{split}
Y(t,\theta) &=Y_{0}(\theta)+\int^{t}_{0}[b_{X}(X(s,\theta),\theta)Y(s,\theta)+b_{\theta}(X(s,\theta),\theta)]\td s \\ &+\int^{t}_{0}[\sigma_{X}(X(s,\theta),\theta)Y(s,\theta)+\sigma_{\theta}(X(s,\theta),\theta)]\td W_{s}.
\end{split}
\end{equation}
Notice that this equation depends on \eqref{Xt} and the unknown parameters, $\theta$. The dependence on $\theta$ implies that one cannot use this method for a priori sensitivity analysis. However, it can be used for a posteriori sensitivity analysis, evaluated at the estimates, to verify the difficulties of the particular parameter regimes. 

As an empirical method for a priori sensitivity analysis, we simply check the deviation of $\hat{k}(t,\theta)$ as a function of  the training collocation nodes, $\theta\in \Theta$, which are available from Step~2 of Algorithm~\ref{alg:surrogate}. Of course, a more elaborate global sensitivity analysis technique such as the Sobol index \cite{sobol:93} can be performed as well but we will not pursue this here. From this empirical a priori sensitivity analysis, we found that the corresponding essential statistics, $\mathbb{E}_{\peq}[v(t)v(0)]$ is not sensitive to $x_0$. But it is sensitive to $a$, and strongly sensitive to $\epsilon$. Also, the sensitivity of the low damping case ($\gamma=0.5$) is stronger than the high damping case ($\gamma=0$). To verify the validity of this empirical method, we compute the local sensitivity analysis indices $\bar{y}_\theta := \mathbb{E} [Y(t,\omega)]$, where $Y(t,\omega)$ solves \eqref{Yt} with the true parameters (again this is not feasible in practice since the true parameters are not available).

For our Langevin model, we have the corresponding drift and diffusion coefficients
\begin{equation}
b=(v,-U'(x)-\gamma v)^{\top},\quad \sigma=(0,\sqrt{2\gamma k_{B}T})^{\top}. \nonumber
\end{equation}
Suppose we are interested in the dependence on parameter $x_0$. Introducing the notations $x_{x_{0}}:=\frac{\partial}{\partial x_{0}} x(t,x_{0})$ and  $v_{x_{0}}:=\frac{\partial}{\partial x_{0}} v(t,x_{0})$, the joint system of \eqref{Xt} and \eqref{Yt} (written in differential form) is given as follows,
\begin{equation}\label{extend_x0}
\begin{cases}
\dot{x}=v  \\
\dot{v}=-U'(x;x_0)-\gamma v+\sqrt{2\gamma k_{B}T} \dot{W}, \\
\dot{x}_{x_{0}}= v_{x_{0}} \\
\dot{v}_{x_{0}}=-U''(x;x_{0})x_{x_0}-\gamma v_{x_{0}}-\frac{\partial}{\partial x_{0}} U'(x;x_{0}).
\end{cases}
\end{equation}
There is no stochastic term in the last equation, since the noise in the Langevin model is additive with coefficient independent of $x_0$. Notice that $(x_{x_0}, v_{x_0})\equiv(1,0)$ is a fixed point for the last two equations in \eqref{extend_x0} because of
\begin{equation}
\begin{split}
\frac{\partial}{\partial x_{0}} U'(x;x_{0})&=a\frac{\partial}{\partial x_{0}} U'_{0}(a(x-x_0))\\
&=-a^2U''_{0}(a(x-x_0))=-\frac{\partial}{\partial x} U'(x;x_{0}). \nonumber
\end{split}
\end{equation}
As a result, by choosing the initial condition to be $(x(0),v(0))=(x_0,c)$ for an arbitrary constant $c$, one can claim that the solution $v(t)$ of the Langevin model \eqref{Lan_sys} is independent of $x_0$. Thus, under this circumstance $\bb{E}_{p^{\dagger}_{eq}}[v(t)v(0)]$ is independent of $x_0^{\dagger}$, confirming the conclusion from our empirical a priori sensitivity analysis. 

While this parameter insensitivity may seem discouraging, we can use it to our advantage to simplify the problem. That is, we can assign an arbitrary value to $x_0$ in applying the Algorithm \ref{alg:surrogate} and consider estimating only $\theta=(a,\epsilon)$. Once these two parameters are estimated, we can use the formula in \eqref{axo} to estimate $x_{0}$. 

We also use this local sensitivity analysis to verify the validity of the a priori sensitivity analysis with respect to $\epsilon$ and $a$. Since it is difficult to analyze explicit behaviors as above, for a given realization of the Langevin model, we solve the remaining ODEs by the fourth order Runge-Kutta method. Figure \ref{extended_sys} shows the numerical results of $\bar{v}_a$ and $\bar{v}_{\epsilon}$, where the average is estimated over 3000 realizations of the Langevin model. Based on the scales of $\bar{v}_a$ and $\bar{v}_{\epsilon}$ in Figure \ref{extended_sys}, we confirm the validity of the sensitivity of the two-point statistics $\bb{E}_{p^{\dagger}_{eq}}[v(t)v(0)]$ that was found empirically. Namely, the identifiability of $\epsilon$ is stronger than $a$ and the identifiability of the low damping case ($\gamma=0.5$) is stronger than the high damping case ($\gamma=5.0$). 

\begin{figure}[ht]
\centering
\includegraphics[width=.5\textwidth]{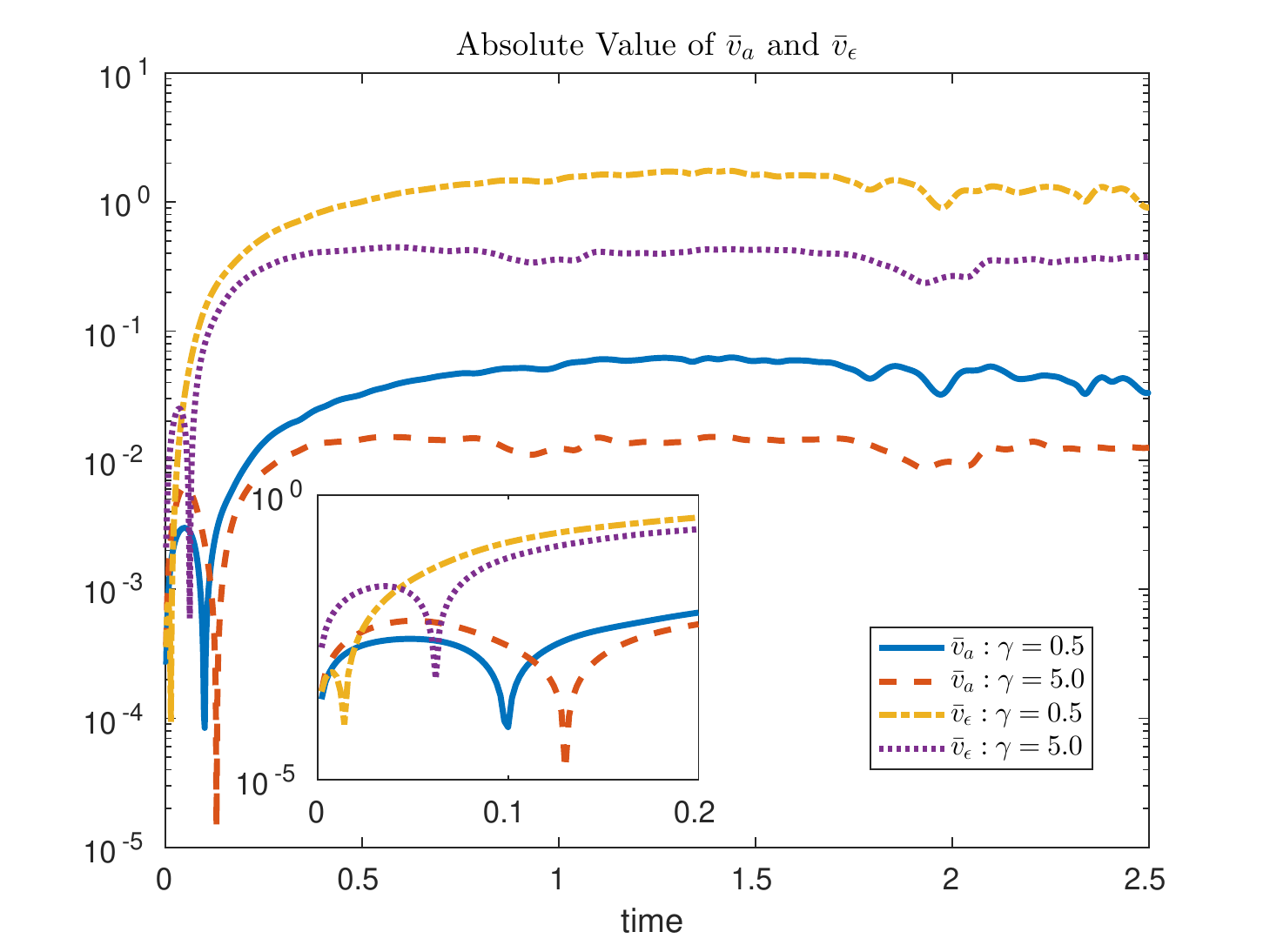}
\caption{The absolute value of $\bar{v}_{a}$ and $\bar{v}_{\epsilon}$ for both $\gamma=0.5$ and $\gamma=5.0$.}
\label{extended_sys}
\end{figure}

\subsection{Numerical Results}
We now present several numerical results. Based on a time series of size $10^7$ (with temporal lag $h=2\times 10^{-3}$), we compare the conventional method and the proposed scheme that fits the essential statistics under $\gamma^{\dagger}=0.5$ and $5.0$, representing a low damping case and a high damping case, respectively. The true values of the remaining four  parameters  are set to $(k_BT^{\dagger},\epsilon^{\dagger},a^{\dagger},x^{\dagger}_0)=(1,0.2,10,0)$. The data, in the form of a time series, are generated from the  model \eqref{psys} using an operator-splitting method \cite{Adam:17}, which is a strong second-order method. 

\subsubsection{Low Damping Case}
We start with the conventional method. To examine the sensitivity of the estimation for $(\epsilon^{\dagger},a^{\dagger}, x^{\dagger}_0)$ to the estimates $\hat{T}$, as a comparison, we also include the estimation in which we use the true value $k_BT^{\dagger}=1$ in solving \eqref{match_eq1}. (we name the results of this scenario as partial estimates)
\begin{table}[ht]
\begin{center}
\begin{tabular}{c|c|c|c|c|c}
\hline
& $k_{B}T$ & $\gamma$ & $\epsilon$ & $a$ & $x_0$ \\
\hline
True & 1.0000 & 0.50000 & 0.20000 & 10.000 & 0.0000 \\
\hline
Eq stat full estimates & 0.99903 & 0.50003 & 0.16870 & 10.888 & 0.001120\\
\hline
Eq stat partial estimates & - & 0.50003 & 0.20039 & 10.005 & 0.011263\\
\hline \hline
Essential stat estimates & 0.99903 & 0.50003 & 0.20004 & 9.9741 & 0.008265\\
\hline
\end{tabular}
\caption{Full and partial estimates of the conventional method using the equilibrium statistics (above) and the estimates using essential statistics (below): the low damping case.}
\label{Tab:Classical}
\end{center}
\end{table}
The results have been listed in Table \ref{Tab:Classical}. We clearly observe that the results of the conventional method is sensitive to the value of $\hat{T}$. To explain the sensitivity, recall that in solving \eqref{match_eq1}, $a$ and $x_0$ are viewed as functions of $\epsilon$.  This suggests that the high sensitivity would occur if
\begin{equation}
\frac{\partial \epsilon}{\partial k_BT}=\left(\frac{\partial \mathbb{E}_{p_{eq}}[x^{3}; T,\epsilon,a(\epsilon),x_0(\epsilon)]}{\partial \epsilon}\right)^{-1}\cdot \frac{\partial \mathbb{E}_{p_{eq}}[x^{3};T,\epsilon,a(\epsilon),x_0(\epsilon)]}{\partial k_BT} \gg 1, \nonumber
\end{equation}
where $a(\epsilon)$ and $x_0(\epsilon)$ are given by \eqref{axo} and the implicit function theorem has been used. By direct computation, we get
\begin{equation}
\frac{\partial}{\partial k_{B}T} \mathbb{E}_{p_{eq}}[x^{n}]=\frac{1}{(k_{B}T)^{2}}\left(\mathbb{E}_{p_{eq}}[x^n U(x)]-\mathbb{E}_{p_{eq}}[x^n]\mathbb{E}_{p_{eq}}[U(x)]\right),\quad n=1,2,\dots.
\nonumber
\end{equation}
Since the explicit formula for the partial derivative with respect to $\epsilon$ is not easy to compute, its value will be computed using a finite difference formula. Evaluating it at $(k_B T, \epsilon, a,x_0)=(1, 0.2, 10, 0)$ we have
\begin{equation*}
\begin{split}
\frac{\partial }{\partial k_BT} \mathbb{E}_{p_{eq}}[x^{3}; T,\epsilon,a(\epsilon),x_0(\epsilon)]&=9.034;\\ \frac{\partial }{\partial \epsilon} \mathbb{E}_{p_{eq}}[x^{3};T,\epsilon,a(\epsilon),x_0(\epsilon)]&=0.06396,
\end{split}
\end{equation*}
which supports our observation.

\smallskip

Next we report the estimations obtained from the new approach using essential statistics, and the results are listed in Table \ref{Tab:Classical} (below). The improvement in accuracy for $\epsilon^{\dagger}$ and $a^{\dagger}$ are noticeable compared with the full estimates of the conventional method. Here we used a total of $20$ essential statistics given by $k_A(t_{i};\theta^{\dagger})$ for $t_i=0.1i$, $i=1,2,\dots,20$. We applied Algorithm~\ref{alg:surrogate} to solve an ensemble of two-dimensional ($a$ and $x_0$) nonlinear least-squares problems based on $300$ uniformly generated random initial conditions. Figure \ref{contour:low} (left) shows the contour plot of the cost function together with all of the estimates. Notice that except for few outliers, most of the estimates lie along the low value of the contour of the cost function. The estimate $\hat{\epsilon}$ reported is the average of all the estimates (excluding the outliers). To satisfy the equilibrium constraint, $a(\hat{\epsilon})$ given by \eqref{axo} is used as the estimates of $a^{\dagger}$.

\begin{figure}
\centering
\includegraphics[width=.45\textwidth]{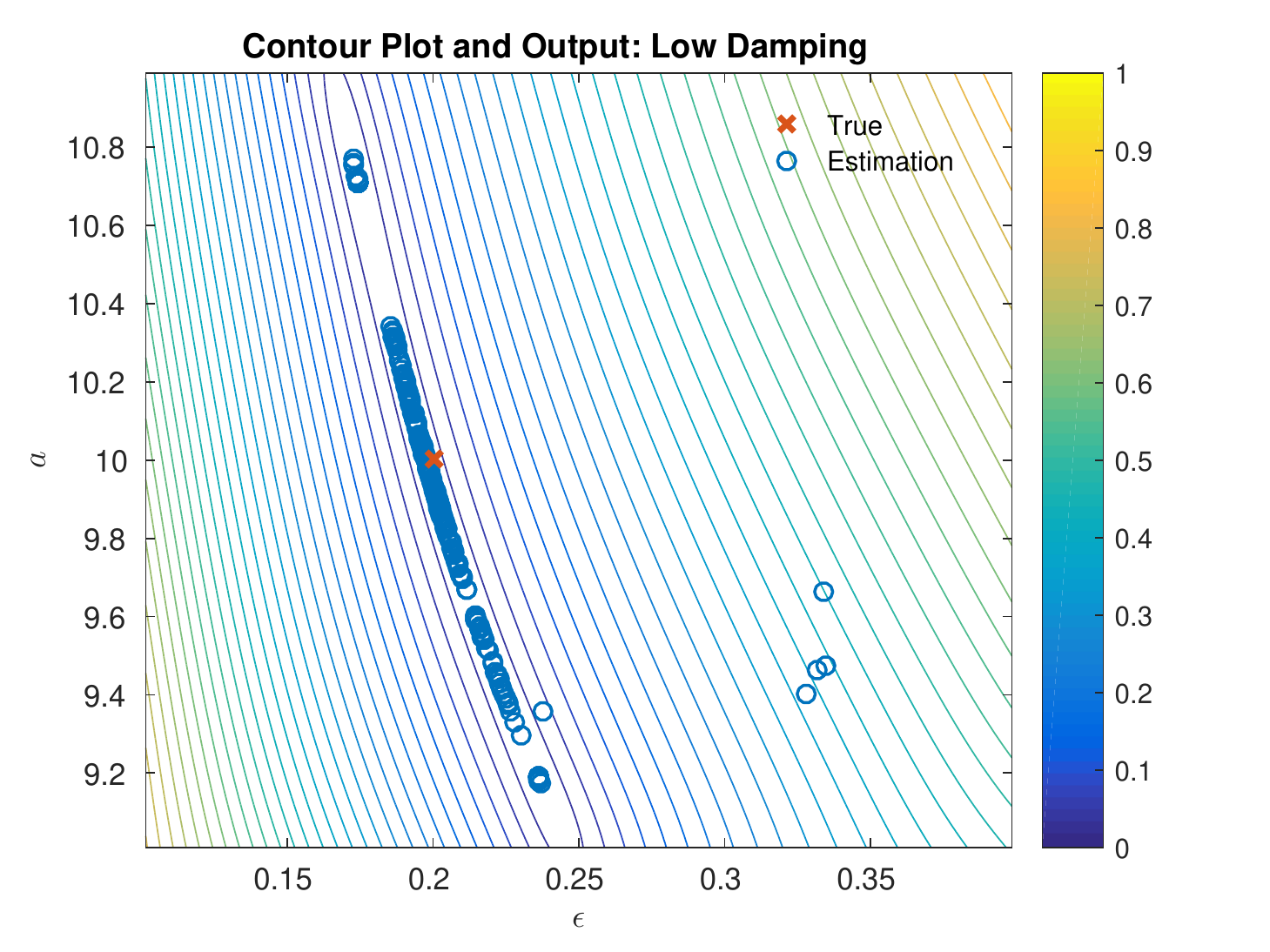}
\includegraphics[width=.45\textwidth]{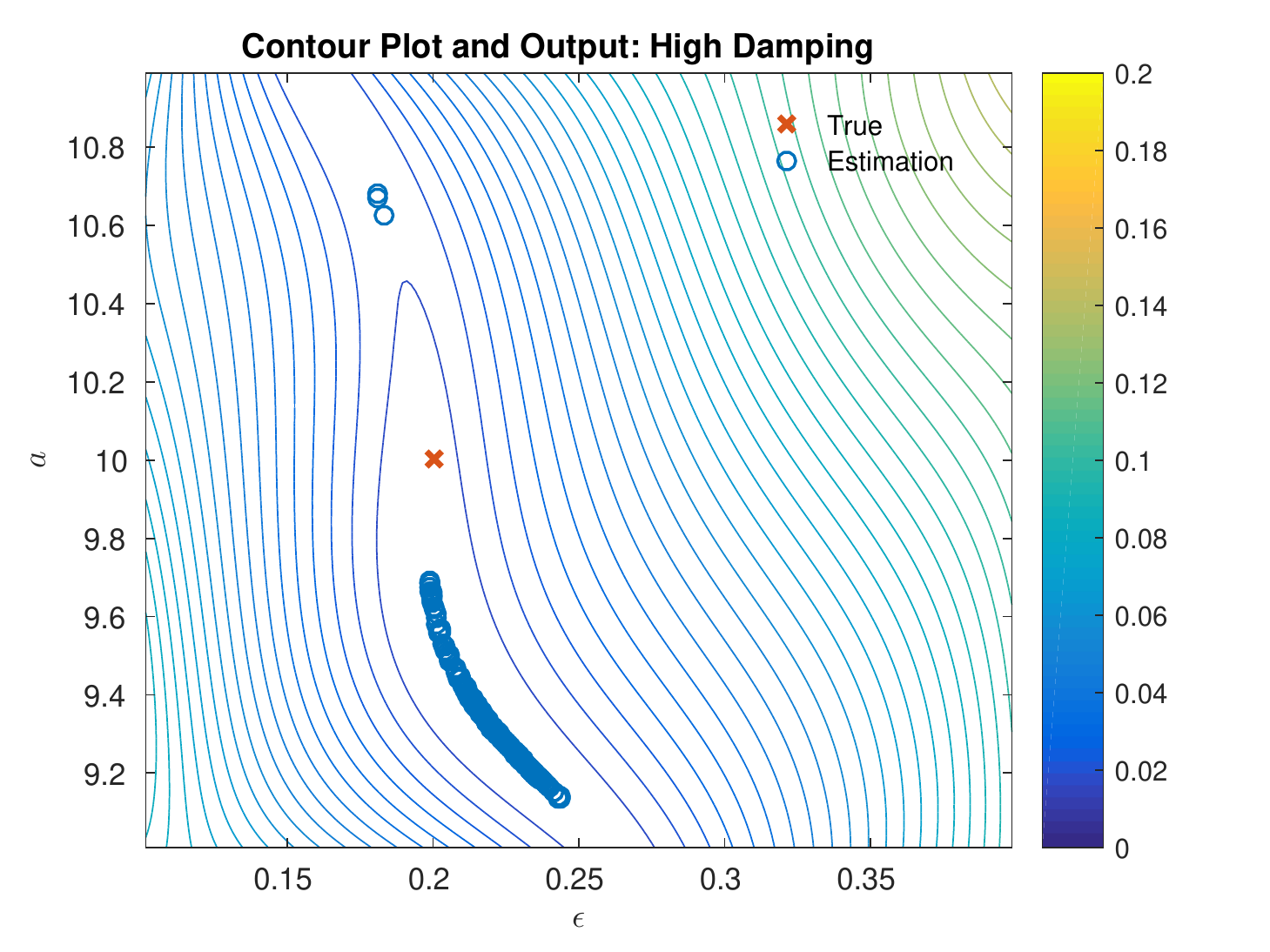}
\caption{The contour plot of the cost function together with the estimates: low damping case (left) and high damping case (right).}
\label{contour:low}
\end{figure}

The results from this test indicate that for the conventional method, 
it is difficult to obtain very accurate estimates for $\epsilon^{\dagger}$, unless $ T^{\dagger}$ can be estimated accurately, e.g., by using longer time series.  In contrast, the proposed approach of using essential statistics is much  less sensitive to the error in $\hat{T}$.  This approach, however, requires a pre-computing step as noted in Algorithm~\ref{alg:surrogate}. In our numerical test, we solved the model \eqref{Lan_sys} on 64 collocation nodes to evaluate the essential statistics over different values of $(\epsilon, a)$, that is, an order $M_C=8$ Chebyshev nodes were used to construct the $\Theta$ used in \eqref{collocation}.  However, as we have previously alluded to, this can be done in parallel, and it will not become a serious issue as long as the dimension of the parameter is relatively small. As for the value of $M$ in \eqref{collocation} we picked $M=6$. (Same scheme was applied to the high damping case)

Another interesting issue arises when the damping parameter is large. Since the  conventional method fully relies on one-point statistics with respect to equilibrium density, it is important to have high-quality independent samples.  In Figure (left) \ref{time correlation: xv}, we show the time correlation of $x$ for both the low damping regimes ($\gamma=0.5$) and high damping regime ($\gamma=5.0$). We observe that in the latter case, the auto-correlation function of $x$ decays much slower, indicating a strong correlation among the samples with small lags. In this case, the estimates from the conventional method will deteriorate due to the difficulty in obtaining high quality independent sampling.

\begin{figure}[ht]
\centering
\includegraphics[width=.45\textwidth]{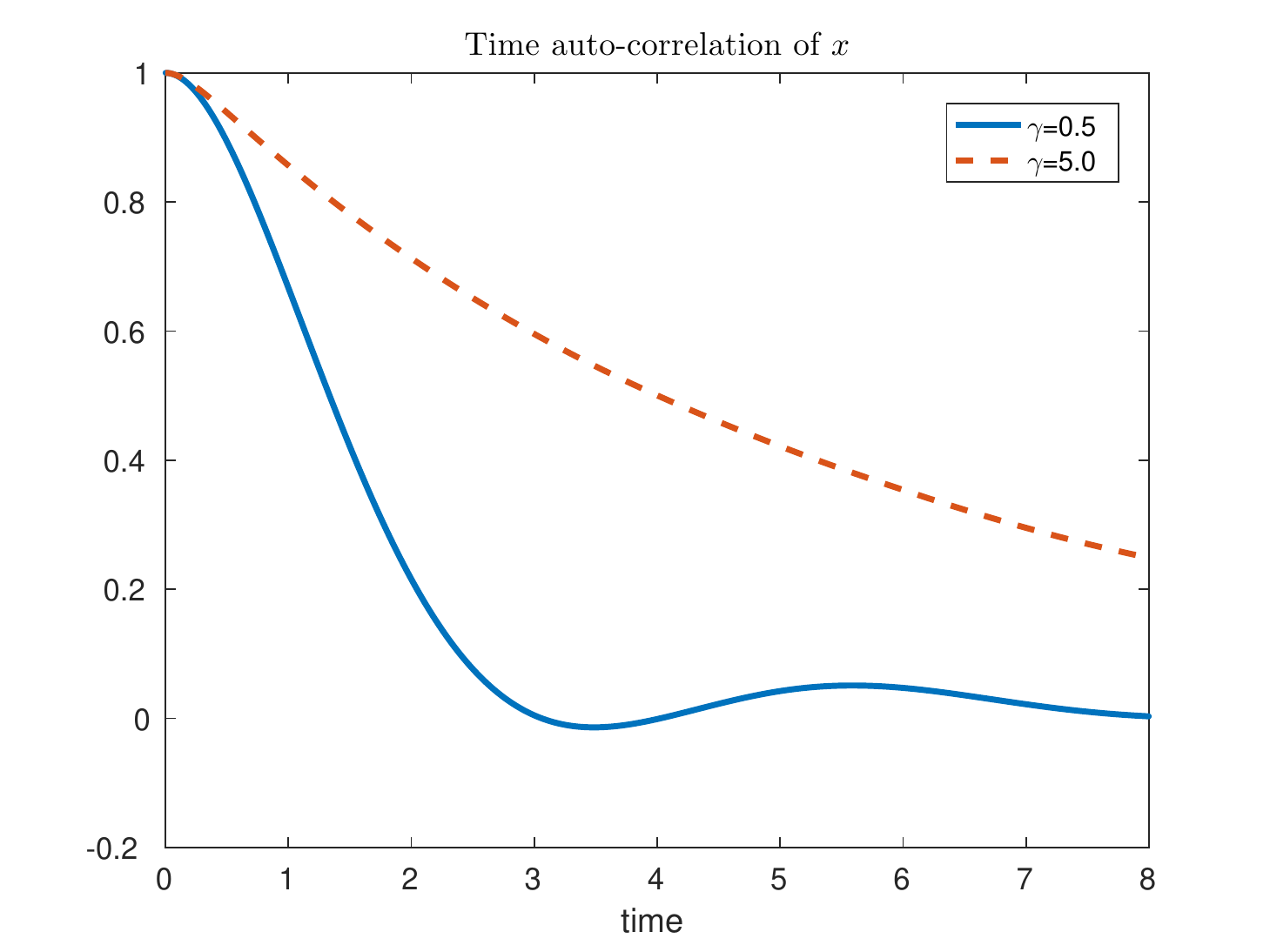}
\includegraphics[width=.45\textwidth]{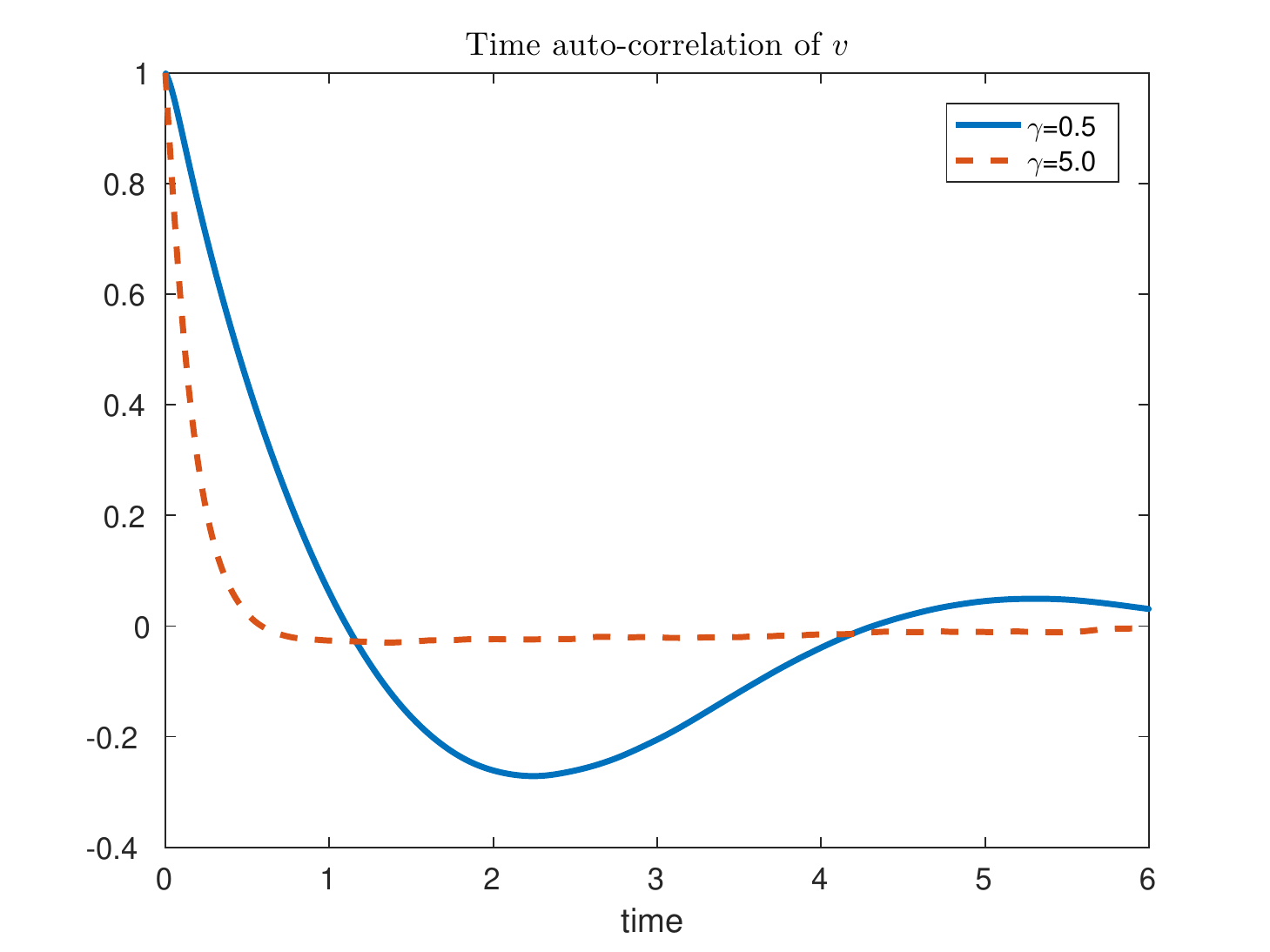}
\caption{The time auto-correlation of $x$ (left) and $v$ (right).}
\label{time correlation: xv}
\end{figure}

\subsubsection{High Damping Case}

In this section, we focus on the high damping regime $\gamma^{\dagger}=5.0$. We will also make a connection between Figure \ref{extended_sys} and the estimation using the essential statistics. As shown in Figure \ref{time correlation: xv} (left), we no longer have a fast decay of the time auto-correlation of $x$. As a result, without changing the sample size, significant error will occur in estimating the moment of $x$ used in the conventional method.
\begin{table}[ht]
\begin{center}
\begin{tabular}{c|c|c|c|c|c}
\hline
& $k_{B}T$ & $\gamma$ & $\epsilon$ & $a$ & $x_0$  \\
\hline
True & 1.0000 & 5.0000 & 0.20000 & 10.000 & 0.0000 \\
\hline
Eq stat full estimates  & 0.99949 & 4.9993 & 0.69579 & 5.4717 & 0.12990\\
\hline
Eq stat partial estimates & - & 4.9993 & 0.25931 & 8.9295 & -0.0034460\\
\hline \hline
Essential stat estimates & 0.99949 & 4.9993 & 0.21160 & 9.8663 &-0.01822\\
\hline
\end{tabular}
\caption{Full and partial estimates of the conventional method using the equilibrium statistics (above) and the estimates using essential statistics (below): the high damping case.}\label{Tab:classical_high}
\end{center}
\end{table}
This can be clearly seen from the results listed in Table \ref{Tab:classical_high}(above). In particular, besides suffering from the sensitivity to $\hat{T}$, the error in estimating $\mathbb{E}_{p^{\dagger}_{eq}}[x^n]$ for $n=1,2,3$ also leads to inaccurate estimates for $\epsilon^{\dagger}$, $a^{\dagger}$ and $x^{\dagger}_0$. We should point out an interesting fact, that is, although the marginal distribution of $x$ at the equilibrium state is independent of $\gamma$, a large value of $\gamma$ causes difficulties in estimating the moments of $x$ in practice.

Table \ref{Tab:classical_high} and Figure \ref{contour:low} (right) show the numerical results using the essential statistics. Although the relative error is not as small as in the low damping case, it is still lower than the estimates from the conventional method. Here is how our method is implemented: Similar to the low damping case, \eqref{axo} is used in estimating the value of $a^{\dagger}$ based on the estimates $\hat{\epsilon}$. The 20 essential statistics here are given by $k_A(t_{i};\theta^{\dagger})$ for $t_i=0.04i+0.2$, $i=1,2,\dots,20$ (suggested by the time auto-correlation of $v$ shown in Figure \ref{time correlation: xv} (right)), which are in a much shorter time interval. Compared to the low damping case, the loss of accuracy can be verified based on the sensitivity analysis (Figure \ref{extended_sys}), which indicated that the parameters $(a,x_0)$ in this regime are less identifiable when $\gamma$ is large.

\section{Example II: A Gradient System with a Triple-Well Potential} \label{sec5}

Our next example  is a gradient system driven by white noise.   We consider a two-dimensional stochastic system as follow, 
\begin{equation}\label{triple}
\td \textbf{x}(t)= -C\nabla V(\textbf{x}) \td t+ \sqrt{2k_{B}T}\td\textbf{W}_{t}, 
\end{equation}
where $\textbf{W}_t$ is a two-dimensional Wiener process, $V$ is a potential energy, and $C$ is a matrix defined by 
\begin{equation}
C=\begin{pmatrix}
1 & -d \\
d & 1 \\
\end{pmatrix},\quad d\in(-1,1). \nonumber
\end{equation}
The generator $\mathcal{L}$ of the system \eqref{triple} is given by,
\begin{equation} \label{generator_well}
\mathcal{L}f=-(V_{x_{1}}-d V_{x_{2}})\frac{\partial f}{\partial x_{1}}-(d V_{x_{1}}+V_{x_{2}})\frac{\partial f}{\partial x_{2}}+k_{B}T \Delta f.
\end{equation}
We choose a triple-well potential function $V$ similar to the model in \cite{Hannachi:01},
\begin{eqnarray}\label{potential_well}
V(x_{1},x_{2}) &=&-v(x_{1}^{2}+x_{2}^{2})-(1-\gamma)v\big((x_{1}-2a)^{2}+x_{2}^{2}\big)-(1+\gamma)v\big((x_{1}-a)^{2}+(x_{2}-a\sqrt{3})^{2}\big)\nonumber \\  &&+0.2[(x_{1}-a)^{2}+(x_{2}-a/\sqrt{3})^{2}],
\end{eqnarray}
with
\begin{equation}
v(z)=10 \exp\left(\frac{1}{z^{2}-a^{2}}\right)\cdot \chi_{(-a,a)}(z),\quad z\in \bb{R},\nonumber
\end{equation}
where $\chi_{(-a,a)}(z)$ denotes the characteristic function over the interval $(-a,a)$. Notice that the matrix $C$ is positive definite. The additional quadratic term $0.2[(x_{1}-a)^{2}+(x_{2}-a/\sqrt{3})^{2}]$ in the triple-well potential \eqref{potential_well} is, again, a smooth retaining potential. It is well known \cite{Pavliotis:16} that the triple-well model \eqref{triple} yields an equilibrium distribution given by
\begin{equation} \label{triple_peq}
p_{eq}(\textbf{x}) \propto \exp\left({-\frac{V(\textbf{x})}{k_{B}T}}\right) \quad \textbf{x}\in \bb{R}^{2},
\end{equation}
which is independent of parameter $d$, the off-diagonal element of $C$.

In the numerical tests, we set $(d^{\dagger},a^{\dagger},k_B T^{\dagger},\gamma^{\dagger})=(0.5,1,1.5,0.25)$ as the true values of the parameters. Figure \ref{potential_scatter} shows the contour plot of the potential and the scatter plot of the time series under this set of parameters. To generate the data from \eqref{triple}, we applied the weak trapezoidal method introduced in \cite{Anderson:09}, which is a weak second-order method. {  In the remainder of this section, we will choose appropriate essential statistics from which $\{d,T\}$ can be directly estimated and the remaining parameters, $\{a,\gamma\}$, will be estimated using an appropriate least-square problem.}

\begin{figure}[ht]
	\centering
	\includegraphics[scale=0.4]{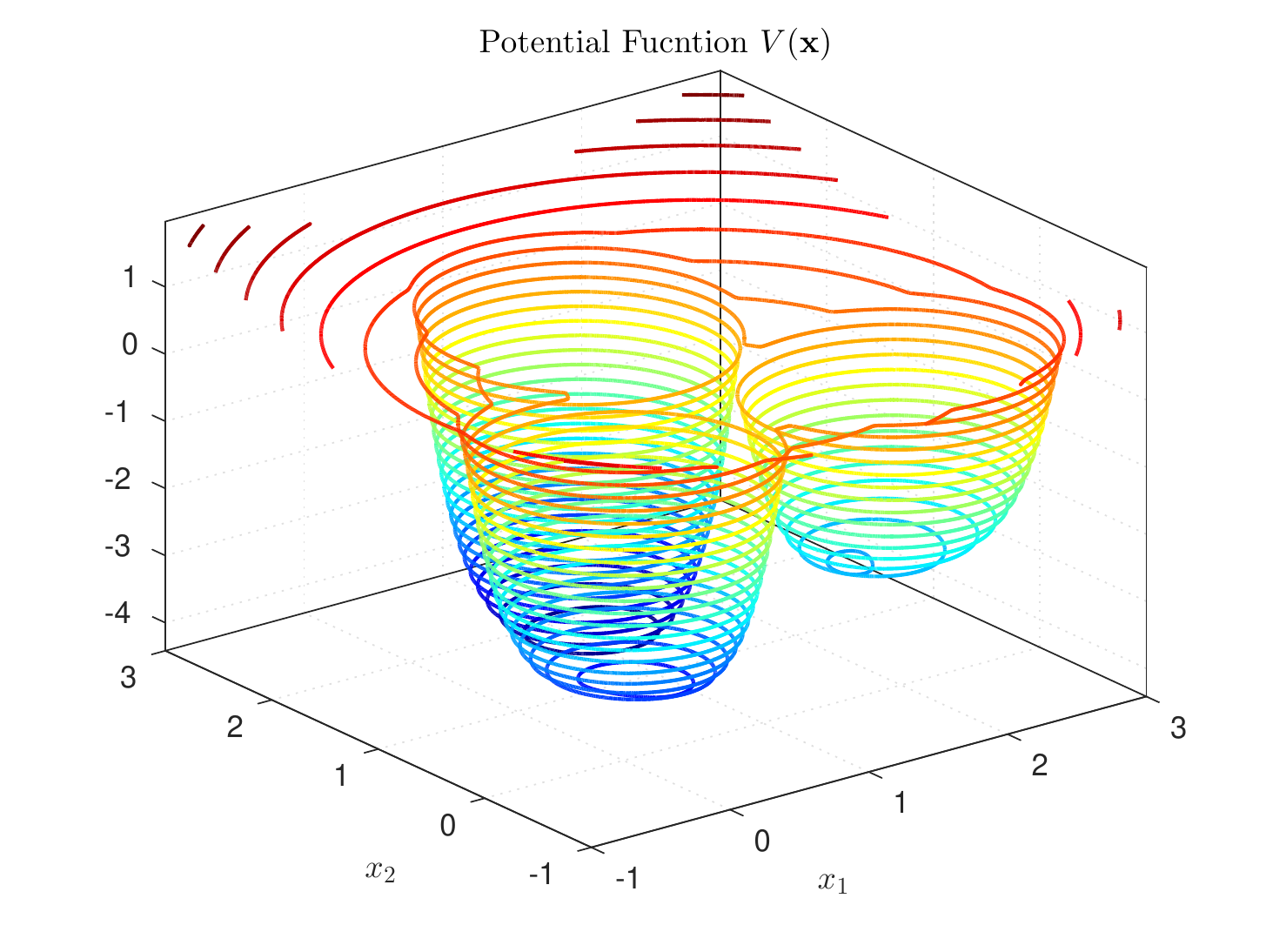}
	\includegraphics[scale=0.4]{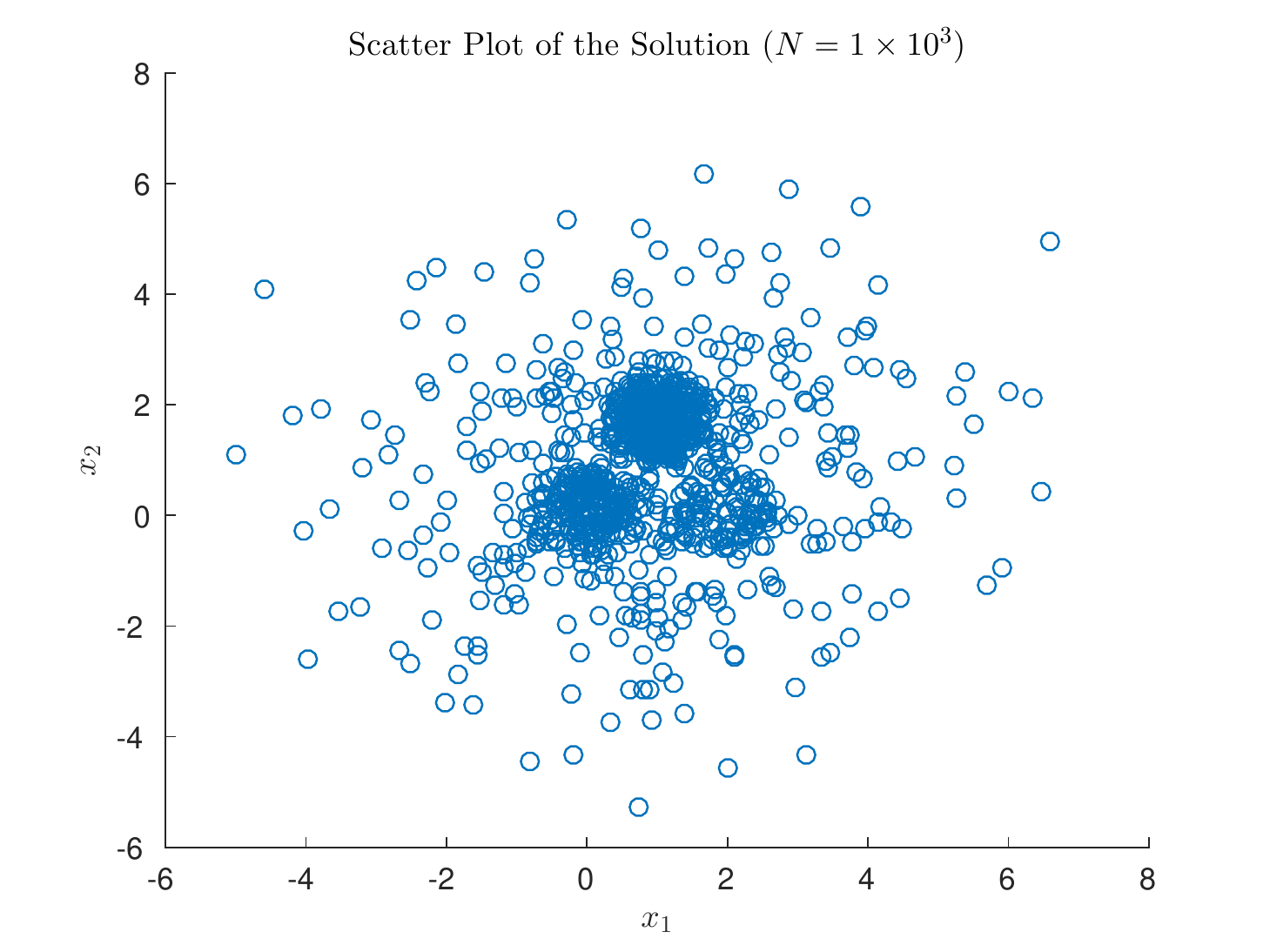}
	\caption{The contour plot of the potential (left) and the scatter plot of the solution with $10^3$ sample points (right).}
	\label{potential_scatter}
\end{figure}

\subsection{The Essential Statistics and Reduction of the Parameter Space}

Similar to the Langevin model, we consider an external forcing that is constant in $x$. Subsequently, the perturbed dynamics is given by,
\begin{equation}\label{per_triple}
\td \textbf{x}(t)= (-C\nabla V(\textbf{x})+ f(t) \delta ) \td t+ \sqrt{2k_{B}T}\td \textbf{W}_{t},
\end{equation}
where $|\delta| \ll 1$. If we select $A(\textbf{x}):=\textbf{x}$ as the observable, the corresponding linear response operator reads
\begin{equation}
k_{A}(t;\theta)=\mathbb{E}_{p_{eq}}[A(\textbf{x}(t))\otimes B(\textbf{x}(0))], \nonumber
\end{equation}
where $B=(B_{1},B_{2})^{\top}$ with
\begin{equation}
B_{i}(x_1,x_2)=\frac{1}{k_{B}T}\frac{\partial}{\partial x_i}V(x_1,x_2), \quad i=1,2, \nonumber
\end{equation}
given by \eqref{RA}. As a result, the entries of the linear response operator $k_{A}(t;\theta)$ satisfy
\begin{equation}\label{potential-esst}
k_{(i,j)}(t;\theta):=\frac{1}{k_{B}T}\mathbb{E}_{p_{eq}}[x_{i}(t)V_{x_{j}}(x_{1}(0),x_{2}(0))], \quad  i,j=1,2, 
\end{equation}
where $V_{x_{i}}$ denotes the partial derivative of $V$ with respect to $x_{i}$. Using integration by parts one can show that
\begin{equation}
\begin{split}
\int_{\bb{R}} x_{i}V_{x_i}(x_1,x_2)\exp(-V(x_1,x_2)/k_{B}T)\td x_j =
\delta_{ij}k_B T \int_{\bb{R}} \exp(-V(x_1,x_2)/k_{B}T)\td x_j, \quad i,j=1,2, \nonumber
\end{split}
\end{equation}
which leads to
\begin{equation}\label{triple_T}
k_{i,j}(0;\theta)=\delta_{ij}.
\end{equation}
This is known as the equipartition of the energy in statistical mechanics.

\medskip

Notice that the response operators $k_{(i,j)}(t;\theta^\dagger)$ are not accessible due to the fact that the function $V_{x_i}(\textbf{x},\theta^\dagger)$ depends on unknown true parameter values $(k_B T^\dagger,a^\dagger,\gamma^\dagger)$. This is precisely one of the issue raised in Remark~\ref{remark}. For gradient flow systems, this problem can be overcome by introducing a linear transformation to another set of two-point statistics. We define,
\begin{equation}\label{m_t}
m_{i,j}(t;\theta):=\bb{E}_{p_{eq}}[x_{i}(t)x_{j}(0)], \quad i,j=1,2,
\end{equation}
and consider their time derivatives. Following the same calculations that led to \eqref{diff_with_gen}, we obtain the following identities for $t>0$
\begin{equation}\label{mt}
\frac{\td }{\td t}m_{i,j}(t;\theta)=k_{B}T\big(-k_{j,i}(-t;\theta)+(-1)^{i-1}d\cdot k_{j,3-i}(-t;\theta)\big), \quad i,j=1,2. 
\end{equation}
It is more helpful to rewrite \eqref{mt} into a linear system for $t>0$
\begin{equation}\label{linear_sys}
k_BT
\begin{pmatrix}
-1 & d & 0 & 0 \\ 0 & 0 & -1 & d \\ -d & -1 & 0 & 0 \\ 0 & 0& -d &-1 
\end{pmatrix}
\begin{pmatrix}
k_{1,1}(-t) \\ k_{1,2}(-t) \\ k_{2,1}(-t) \\ k_{2,2}(-t)
\end{pmatrix}=
\begin{pmatrix}
m'_{1,1}(t) \\ m'_{1,2}(t) \\ m'_{2,1}(t) \\ m'_{2,2}(t)
\end{pmatrix},
\end{equation}
where the coefficient matrix is non-singular since $d\in(-1,1)$. This linear relationship suggests that one can consider fitting $m_{i,j}$ in placed of $k_{i,j}$, since the former is numerically accessible.

Let $t\rightarrow 0^{+}$ and apply \eqref{triple_T} to the first and the third equations of \eqref{linear_sys} under $\theta=\theta^{\dagger}$, we obtain the following estimates for $d^{\dagger}$ and $T^{\dagger}$
\begin{equation} \label{reduction_triple}
k_{B}T^{\dagger}=-m'_{1,1}(0^{+};\theta^{\dagger}); \quad d^{\dagger}=-m'_{2,1}(0^{+};\theta^{\dagger})/k_{B}T^{\dagger},
\end{equation}
where $m'_{i,j}(0^{+};\theta^{\dagger})$ are computable from the available sample. Equations~\eqref{reduction_triple} reduces the problem into a two-dimensional problem of estimating $(a^{\dagger},\gamma^{\dagger})$ only.

\subsection{Numerical Results}

As we have pointed out earlier, the linear response operator, $k_{(i,j)}(t;\theta^{\dagger})$ \eqref{potential-esst}, are not directly accessible. Motivated by the linear relation in \eqref{linear_sys}, we consider to infer the values $(a^{\dagger},\gamma^{\dagger})$ from the two-point statistics $m_{i,j}(t;\theta^{\dagger})$ \eqref{m_t}. In particular, we choose $m_{1,1}(t_{i};\theta^{\dagger})$ with $t_{i}=0.1i$, $i=1,\dots,20$, as the essential statistics in our numerical test. 

In light of the fact that  $m_{1,1}(t;\theta)$ is the two-point statistics of $x_1$, we can apply the empirical a priori sensitivity test by checking the dependence of $\hat{k}(t,\theta)$ over the training collocation nodes $\theta=(a,\gamma)\in \Theta$, which are available to us from Step~2 of Algorithm~\ref{alg:surrogate}. In this numerical experiment, the training used 25 collocation nodes, $\theta\in\Theta$. To validate this empirical sensitivity test, we also perform the local sensitivity analysis as described in Section~\ref{local_sens_analysis} to parameters $a$ and $\gamma$ (which, again, is not possible in practice since the true parameters are unknown). The results are shown Figure \ref{identify} (left).  From the scales of $\bar{\textbf{x}}_{a}$ and $\bar{\textbf{x}}_{\gamma}$ (average over $2000$ realizations of $\textbf{x}$), we can see that this two-point statistic, $m_{1,1}(t;\theta^{\dagger})$, is quite sensitive to the parameters $\gamma$ and $a$, with stronger dependence on the parameter $a$. 
\begin{figure}[ht]
	\centering
	\includegraphics[scale=0.5]{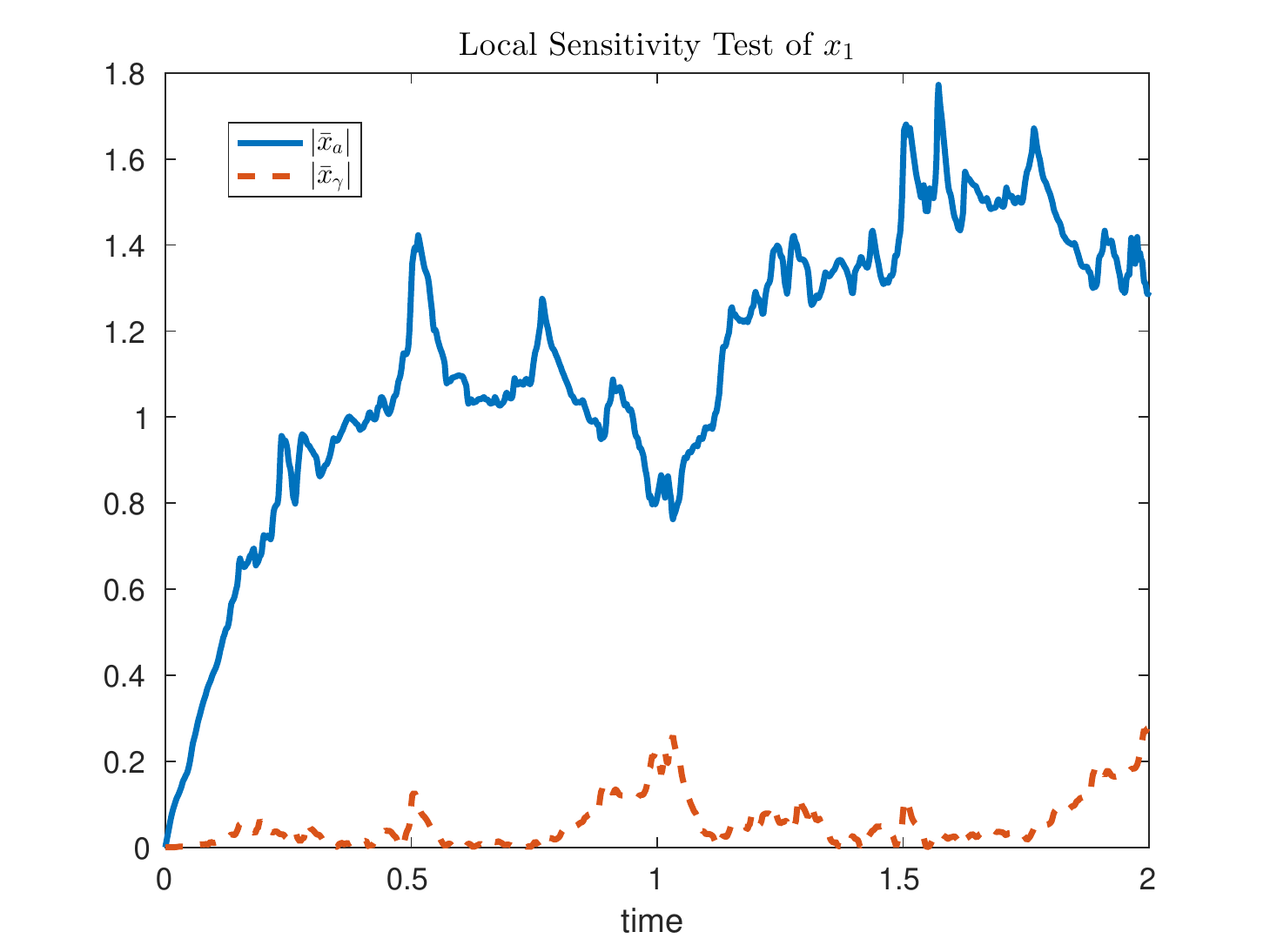}
	\includegraphics[scale=0.5]{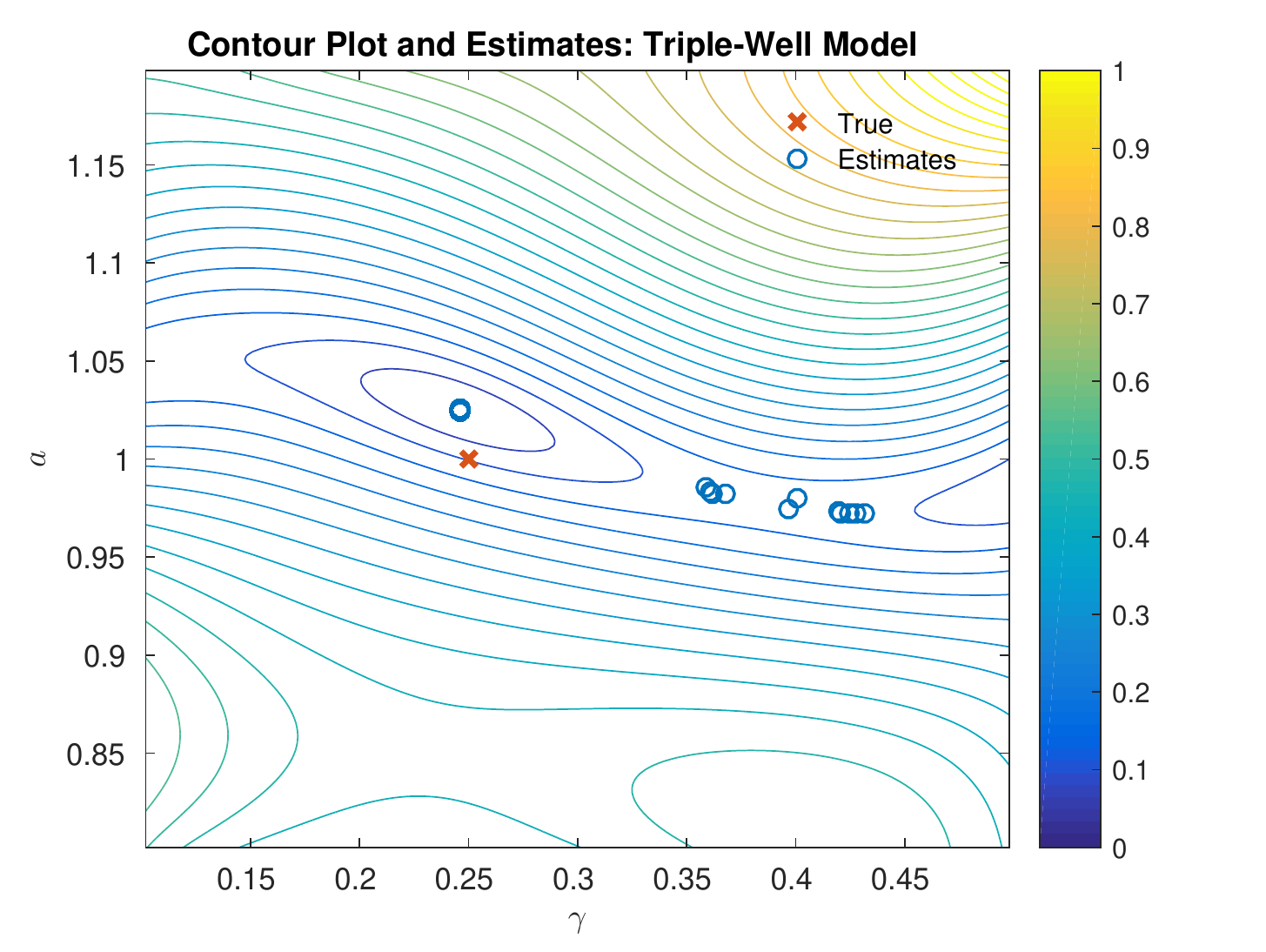}
	\caption{Local sensitivity test of $x_1$ with respect to parameter $a$ and $\gamma$ (left) and the contour plot of the cost function together with the estimates (right).}
	\label{identify}
\end{figure}
Based on a time series of size $4\times 10^6$ (with time lag $h=1\times 10^{-3}$), we implemented Algorithm~\ref{alg:surrogate} with $M=4$ and collocation nodes, $\Theta$, constructed from order $M_C=5$ Chebyshev nodes in \eqref{collocation}. 

\begin{table}[ht]
	\begin{center}
		\begin{tabular}{c|c|c|c|c}
			\hline
			& $a$ & $\gamma$ & $k_{B}T$ & $d$  \\
			\hline
			True & 1.0000 & 0.2500 &1.5000 & 0.5000  \\
			\hline
			Estimates & 1.0249 & 0.2463& 1.4946 & 0.5077 \\
			\hline
		\end{tabular}
		\caption{The estimation results: the Triple-Well model.}
		\label{Tab:triple}
	\end{center}
\end{table}
The estimates are shown in Table \ref{Tab:triple}. We also show in Figure \ref{identify} (right) the contour plot of the cost function together with the estimates based on $300$ uniformly generated initial guesses. Notice that the true parameter value does not lie on the lowest contour value and these discrepancies are due to the surrogate modeling and quality of the samples. However, the estimates $\hat{a}$ and $\hat\gamma$ are still reasonably accurate, as reported in Table \ref{Tab:triple}. We should point out that the estimates $\hat{a}$ and $\hat{\gamma}$ reported in this table are the average of all the estimates (excluding the outliers), while the estimates $k_{B}\hat{T}$ and $\hat{d}$ are obtained by solving \eqref{reduction_triple}. The contour plot also confirms the sensitivity analysis which suggested that estimating $a^{\dagger}$ is easier than $\gamma^{\dagger}$. 

\begin{figure}[ht]
	\centering
	\includegraphics[scale=0.5]{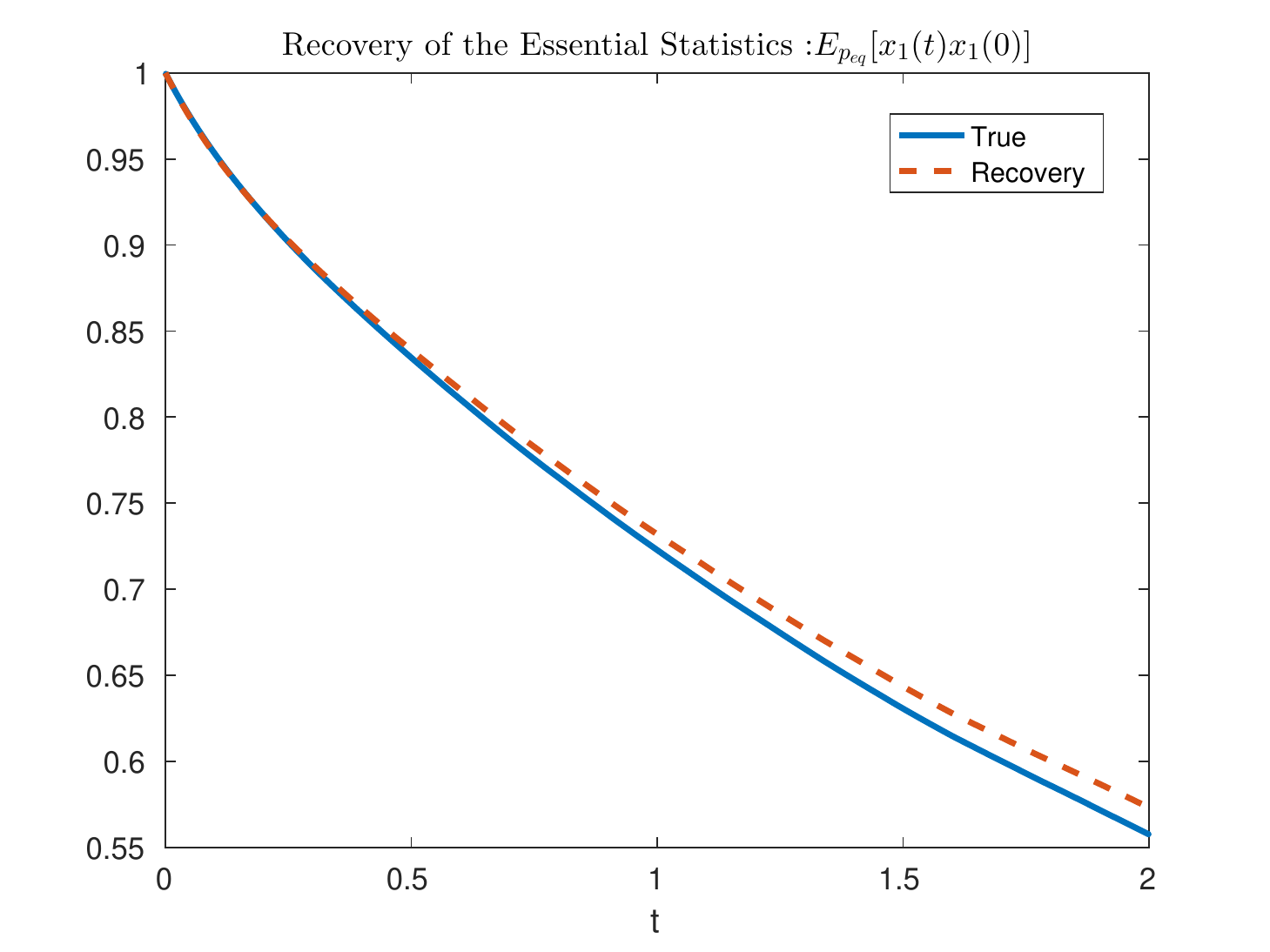}
	\includegraphics[scale=0.5]{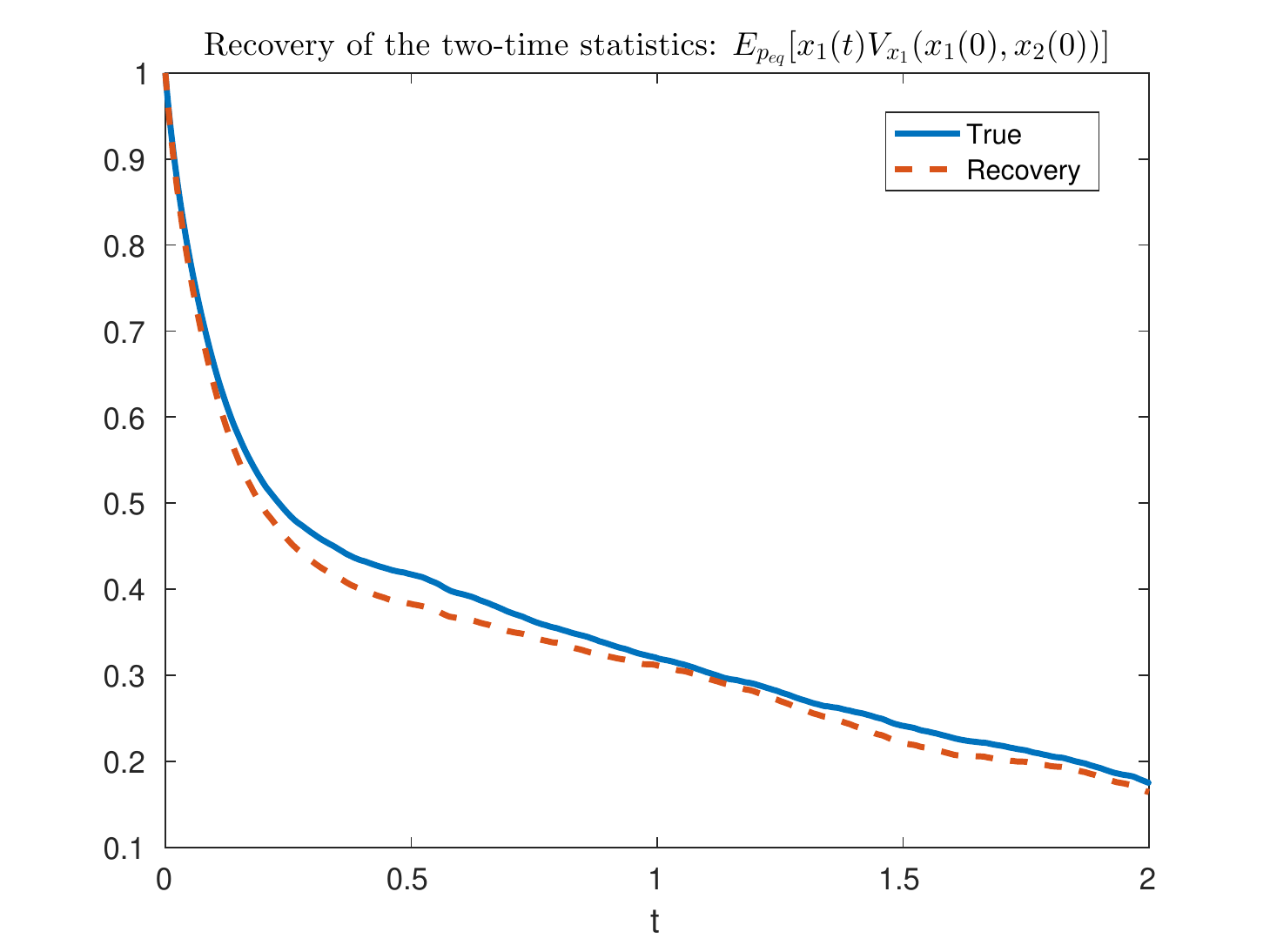}
	\caption{The recovery of the response statistics $\mathbb{E}_{p_{eq}}[x_1(t)x_{1}(0)]$ (left) and the two-point statistics $\mathbb{E}_{p_{eq}}[x_1(t)V_{x_{1}}(x_1(0),x_2(0))]$ (right). All the trajectories have been normalized such that they start at point $(0,1)$.}
	\label{reco}
\end{figure}

One interesting question is whether the approximate values of the parameters can reproduce the essential statistics, which is useful in predicting non-equilibrium averages in the presence of external forces.  Using the estimates reported in Table \ref{Tab:triple}, we compare the resulting two-point statistics used in estimating the parameters, $m_{1,1}(t;\hat\theta)$, with the corresponding true statistics, $m_{1,1}(t;\theta^\dagger)$ (see Figure~\ref{reco}, left). In Figure~\ref{reco} (right), we also compare the estimated response operator, $k_{1,1}(t,\hat\theta)$, and the true response operator, $k_{1,1}(t,\theta^\dagger)$. Excellent agreement is found.

\section{Summary and Further Discussion} \label{sec6}

This paper presented a parameter estimation method for stochastic models in the form of It\^o drift diffusions. { To infer parameters that do not appear in the equilibrium density function, we proposed to fit two-point} {\it essential statistics} formulated using the fluctuation dissipation theory. Building upon the framework established in our previous work \cite{HLZ:17}, we formulated the problem as a nonlinear least-squares problem subjected to a dynamical constraint { when the parameters cannot be estimated directly by solving the corresponding statistical constraints.} To avoid expensive computational cost in evaluating the essential statistics at each iteration when Gauss-Newton method is used, we proposed to solve an approximate least-squares problem based on the polynomial surrogate modeling approach. This approach is motivated by the fact that sampling error cannot be avoided in computing the value of essential statistics and the essential statistics are smooth functions of the parameters. We guaranteed the existence of minimizers of the approximate least-squares problem that converge to the solution of the true least-squares problem that involves the essential statistics under the assumption that these statistics are smooth functions of the parameters. We also showed that the polynomial approximate least-squares problem has a Jacobian that is full rank almost everywhere, which implies the local convergence of Gauss-Newton solutions. We tested the proposed methods on two examples that belong to two large classes of stochastic models --- a Langevin dynamics model and a stochastic gradient system. In general, we expect that the parameter estimation procedure should be carried out as follows.

\begin{enumerate}
	
	\item Reduce the parameter space by direct estimation using appropriate statistics that are easily computable, whenever this is possible.
	
	\item Based on the observable of interest, the functional form of the equilibrium density, the external forcing term, and the available data,
	identify appropriate essential statistics for the remaining parameters using a priori sensitivity analysis test. In our implementation, we compare the essential statistics computed on the training parameter values  (e.g., collocation nodes).  A more elaborate global sensitivity analysis technique such as the Sobol index \cite{sobol:93} can be used to determine the parameter identifiability.
	
	\item For the remaining parameters, formulate a nonlinear least-squares problem using the appropriate essential statistics at some training parameter values.
	
	\item Apply algorithm \ref{alg:surrogate} to obtain the estimates $\hat{\theta}$ of the true parameter values $\theta^{\dagger}$.
	
	\item Apply a sensitivity analysis, such as the local sensitivity analysis as discussed in Section~\ref{local_sens_analysis} as a posteriori confidence check for the estimates. If the parameters are found to be insensitive the selected response statistics, go back to Step 3 and choose a different set of response statistics.
	
\end{enumerate}

One of the restrictions with our formulation is to be able to compute $B^{\dagger}$ in \eqref{hatkA}, which may require knowledge of (some of) the true values $\theta^\dagger$ and this is not feasible in general. As a result, the non-linear system \eqref{nonlin_sys} cannot be evaluated without knowing the true parameters. For the Langevin dynamics, we found that the parameter $k_BT^\dagger$ in $B^\dagger$ can be estimated with the equilibrium statistics, variance of $v$. For the gradient flow problem, there exists an invertible linear transformation between the linear response statistics that are not computable and other two-point statistics that can be estimated without knowing the true parameter values. Since the numerical algorithm \ref{alg:surrogate} does not rely on the FDT formulation, one can apply it on any available two-point statistics, e.g., time auto-correlations of the solution. We are not aware of a general routine that bypasses this difficulty. But as we have demonstrated, at least for general Langevin dynamics models and stochastic gradient systems, this idea will work. 

{  While the proposed method requires the knowledge of the equilibrium density function of the unperturbed system, one can relax this condition and estimate it from the data. For low-dimensional problems, one can use the non-parametric kernel density estimation method \cite{rosenblatt1956remarks,parzen1962estimation}. If the invariant distribution of the dynamical system can be characterized by a density function defined on a smooth manifold (embedded in a high-dimensional phase space) that the data lie on or are close to, then one can estimate the density function by applying the theory of kernel embedding of distribution \cite{smola2007hilbert} on a Hilbert space with basis functions defined on the data manifold. Analogous to the conditional density estimation proposed in \cite{Berry2017MWR,jh:18}, these data-driven basis functions can be obtained via the diffusion maps algorithm \cite{Coifman2006ACHA,Berry2016ACHA}. Another alternative to nonparametric methods is to consider a parametric density function estimation using the moment constrained maximum entropy principle \cite{jaynes:57}. For densities with moderate, say five to seven, dimensions, one can use a recently developed Equation-By-Equation (EBE) method \cite{hh:18} for solving the moment constrained maximum entropy problems. The source code of the EBE scheme is available at \cite{ebecodes}.}

One of our ultimate goals is to use this method for parameter estimation of molecular modeling, in which a Langevin type of model is usually used. The potential energy typically involves a large set of parameters, e.g., the stiffness constants associated with bond stretching and deformation of bond angles,  as well as the coefficients in the van der Waals and screened electro-static interactions. The selection of damping coefficients is also non-trivial, see the review paper \cite{noid2013perspective} for various perspectives. Compared to the existing methods, the novelty of our approach is that we use the response statistics to formulate the parameter estimation problem, which can reveal parameters that do not appear in the equilibrium density. In addition, the polynomial surrogate model provides an efficient mean to solve the nonlinear least-squares problem. 

To achieve this goal, however, there are some remaining challenges. For large parameter space, we suspect that one needs to combine the existing methods to compute the parameters associated with the equilibrium density as one way to reduce the problem (as suggested in Step 1 above) with the proposed method to estimate the remaining parameters. Second, the proposed method requires high quality and possibly large amount of samples for accurate evaluation of the essential statistics. However, since the essential statistics depend on the choice of the observable and external forcing, approximating the high-dimensional integral can be avoided so long as the dimensions of the ranges of these two functions are small. Finally, the underlying model might be subject to modeling error. For example, the model may be derived from a multiscale expansion or just empirically postulated. The formulation of the response statistics in the presence of modeling error will be investigated in separate works.

{\bf Acknowledgments.} This research was partially supported by the NSF Grant DMS-1619661. XL is also supported by the NSF Grant DMS-1522617. JH is also supported by the ONR Grant N00014-16-1-2888.

\bigskip

\appendix
\section{Proof of Theorem \ref{thm:convN}}\label{appA}

In section \ref{subsec:con} we presented the convergence result theorem \ref{thm:convN} and included the proof for one-dimensional case. Here we show the proof for the general $N$-dimensional case. We will first generalize Lemma \ref{lem:p}-\ref{thm:p2} to the $N$-dimensional case.

\begin{lem} \label{lem:p2}
		Let $f\in C^{\ell}([-1,1]^{N})$. For the coefficient of the corresponding least-squares approximation $\alpha_{\vec{k}}$, we have
		\begin{equation*}
		\lim_{\lambda \rightarrow +\infty}   \alpha_{\vec{k}+\lambda \vec{e}_{i}}(k_{i}+\lambda)^{\ell}=0, 
		\end{equation*}
		where $k_i$denote the i$^{\text{th}}$ component of $\vec{k}$ and $\vec{e}_{i}$ is the $i^{\text{th}}$ unit vector.
\end{lem}
\begin{proof}
	To apply Lemma \ref{lem:p}, we introduce the single-variable function
	\begin{equation*}
	\begin{split}
	g(\theta_{i}):= \int_{[-1,1]^{N-1}} f(\theta)\frac{P_{\vec{k}}(\theta)}{P_{k_{i}}(\theta_{i})}\td \mu(\theta_{1})\cdots 
	\td \mu(\theta_{i-1})\td \mu(\theta_{i+1})\cdots \td\mu(\theta_{N}).
	\end{split}
	\end{equation*}
	It is easy to check that $g(\theta_{i})\in C^{\ell}([-1,1])$ and 
	\begin{equation*}
	\int_{-1}^{1}g(\theta_i) P_{k_{i}+\lambda}(\theta) \td \mu(\theta_i) =\alpha^{\vec{j}}_{\vec{k}+\lambda \vec{e}_{i}}.
	\end{equation*}
	Consider the least-squares approximation of $g(\theta_i)$ and apply Lemma \ref{lem:p} to the corresponding coefficients. We are able to draw the conclusion.
\end{proof}

\begin{lem}\label{lem:p3}
	Let $f\in C^{\ell}([-1,1]^{N})$ with $\ell>\frac{3}{2}N$, and consider its least-squares approximation $f^{M}$ given by \eqref{expansion}. Then we have $\|f-f^{M}\|_{\infty}\rightarrow 0$ as $M\rightarrow +\infty$.
\end{lem}
\begin{proof}
	From the proof of Lemma \ref{lem:p1}, it is enough to show that $\{f^{M}\}$ is a Cauchy sequence with respect to $L^{\infty}$-norm. For $m>n$, we have
	\begin{equation*}
	\begin{split}
	\|f^{m}-f^{n}\|_{\infty} &\leq \sum_{\lambda=n+1}^{m} \sum_{\{\vec{k};\|\vec{k}\|_{\infty}=\lambda\}}|\alpha_{\vec{k}}|\|P_{\vec{k}}\|_{\infty} 
	 \leq  \sum_{\lambda=n+1}^{m} \left(\lambda+\frac{1}{2}\right)^{\frac{N}{2}}\sum_{\{\vec{k};\|\vec{k}\|_{\infty}=\lambda\}}|\alpha_{\vec{k}}| ,
	\end{split}
	\end{equation*}
	where we have used the fact that for $\|\vec{k}\|_{\infty}=\lambda$,
	\begin{equation*}
	\begin{split}
	\|P_{\vec{k}}\|_{\infty}&= \left\|\prod_{i=1}^{N} P_{k_i}\right\|_{\infty} 
	=\prod_{i=1}^{N}\left\| P_{k_i}\right\|_{\infty} 
	\leq \prod_{i=1}^{N} \sqrt{k_i+\frac{1}{2}}\leq \left(\lambda+\frac{1}{2}\right)^{\frac{N}{2}}.
	\end{split}
	\end{equation*}
	Further notice the multi-index set $\{\vec{k}; \|\vec{k}\|_{\infty}=\lambda\}$ has total $(\lambda+1)^{N}-\lambda^{N}=O(\lambda^{N-1})$ elements, and by Lemma \ref{lem:p2} we know
	\begin{equation*}
	|\alpha_{\vec{k}}|=o(\lambda^{-\ell}), \quad \lambda=\|\vec{k}\|_{\infty},
	\end{equation*}
	for $\lambda$ large enough. Thus, for large enough $m>n$ we have
	\begin{equation*}
	\|f^{m}-f^{n}\|_{\infty} \leq C \sum_{\lambda=n+1}^{m} \left(\lambda+\frac{1}{2}\right)^{\frac{N}{2}} \lambda^{N-1} \lambda^{-\ell},
	\end{equation*}
	where $C$ is a constant only depends on $N$ and $\ell$. Finally, with $\ell >\frac{3}{2}N$, we know $\{f^{M}\}$ is indeed a Cauchy sequence w.r.t to $L^{\infty}$-norm, which completes the proof.
\end{proof}
To generalize Lemma \ref{thm:p2}, it is enough to consider the partial derivative case, that is, finding the condition for $\|f^{M}_{x_i}-f_{x_i}\|_{\infty}\rightarrow 0$ as $n\rightarrow +\infty$.
\begin{lem}\label{lem:p5}
	For $f\in C^{\ell}([-1,1]^{N}) $ with $\ell>\frac{3}{2}N+2$, consider its least-squares approximation $f^{M}$ given by \eqref{expansion}. Then we have
	\begin{equation*}
	\lim_{M\rightarrow +\infty} \left\|\frac{\partial}{\partial {x_i}}(f^{M}-f) \right\|_{\infty}=0, \quad i=1,2,\dots,N.
	\end{equation*}
\end{lem}
\begin{proof}
	Similar to the proof of Lemma~\ref{lem:p3}, we first show that $\{f^{M}_{x_i}\}$ is a Cauchy sequence. For $m>n$, we have
  \begin{equation*}
	\begin{split}
	\|f_{x_i}^{m}-f_{x_i}^{n}\|_{\infty} &\leq \sum_{\lambda=n+1}^{m} \sum_{\{\vec{k};\|\vec{k}\|_{\infty}=\lambda\}}|\alpha_{\vec{k}}| \left\|\frac{\partial}{\partial_{x_i}}P_{\vec{k}}\right\|_{\infty} \\
	&\leq  \sum_{\lambda=n+1}^{m} \frac{1}{2}\lambda (\lambda+1) \left(\lambda+\frac{1}{2}\right)^{\frac{N}{2}}\sum_{\{\vec{k};\|\vec{k}\|_{\infty}=\lambda\}}|\alpha_{\vec{k}}| ,
	\end{split}
	\end{equation*}
	where we have used the inequality \eqref{eq:pbound}. By Lemma \ref{lem:p2}, similar to the proof of Lemma \ref{lem:p3}, we know $f\in C^{\ell}([-1,1])$ provides
	\begin{equation*}
	\|f_{x_i}^{m}-f_{x_i}^{n}\|_{\infty} \leq C \sum_{\lambda=n+1}^{m} \frac{1}{2}\lambda(\lambda+1) \left(\lambda+\frac{1}{2}\right)^{\frac{N}{2}} \lambda^{N-1} \lambda^{-\ell},
	\end{equation*}
	where $C$ is a constant only depends on $N$ and $\ell$. Finally, with $\ell>\frac{3}{2}N+2$, we able able to conclude that $\{f_{x_i}^{M}\}$ is a Cauchy sequence with respect to $L^{\infty}$-norm. With the same trick  in the proof of Lemma \ref{thm:p2}, one can show that $f_{x_i}^{M}$ indeed converges to $f_{x_i}$ uniformly on $[-1,1]$.
\end{proof}
Now we are ready to prove theorem \ref{thm:convN}.
\begin{proof}
	First, by Lemma \ref{lem:p3} and \ref{lem:p5}, the regularity assumption simply provides
	\begin{equation*}
	\lim_{M \rightarrow +\infty} \|\textbf{f}^{M}-\textbf{f}\|_{\infty}=0, \quad \lim_{M \rightarrow +\infty} \|\nabla \textbf{f}^{M}-\nabla\textbf{f}\|_{\infty}=0.
		\end{equation*}
	To construct a similar function $F^{M}$ used in the proof of proposition \ref{thm:p3}, we introduce the notation
	\begin{equation*}
	\min_{\theta\in [-1,1]^{N}} \sum_{i=1}^{N} (f_{i}^{M})^{2}(\theta)=: \sum_{i=1}^{N} (e^{M}_{i})^{2}, \quad \textbf{e}^{M}:= (e^{M}_{1}, e^{M}_{2},\dots, e^{M}_{N})^{\top},
	\end{equation*}
	and define $F^{M}:= \textbf{f}^{M}-\textbf{e}^{M}$. First, we show the existence of a sequence of solutions, $\{\theta_M^*\}$, that are close to $\theta^\dagger$ by checking the hypothesis of Lemma \ref{lem:p4} to $F^{M}$ at $\theta^{\dagger}$. Since $\nabla \textbf{f}(\theta^{\dagger})$ is invertible, by continuity, for some $r>0$ small enough, there exists a positive constants $\gamma$ such that
	\begin{equation*}
	\| \nabla \textbf{f}(\theta)^{-1} \| \leq \frac{\gamma}{2}, \quad \forall \theta \in B(\theta^{\dagger}, r).
	\end{equation*}
	Later, we will specify $r$ for the convergence of $\theta^*_M\to\theta^\dagger$ as $M\to \infty$.

	Further notice that $\nabla F^{M}= \nabla \textbf{f}^{M}$ and $\nabla \textbf{f}^{M}$ converges to $\nabla\textbf{f}$ uniformly. This implies that  there exists $L_1$ such that $\forall M>L_1$
	\begin{equation}\label{eq:condN1}
	\|\nabla F^{M} (\theta)^{-1} \|\leq \gamma, \quad \forall \theta \in B(\theta^{\dagger}, r).
	\end{equation}
	By the definition of $\textbf{e}^{M}$, we know at $\theta=\theta^{\dagger}$
	\begin{equation*}
	\|F^{M}(\theta^{\dagger})\| = \|\textbf{f}^{M}(\theta^{\dagger})-\textbf{e}^{M}\| \leq 2\|\textbf{f}^{M}(\theta^{\dagger})\|,
	\end{equation*}
	which leads to
	\begin{equation*}
	\lim_{M \rightarrow +\infty} 	\|F^{M}(\theta^{\dagger})\| \leq \lim_{M \rightarrow +\infty} 2\|\textbf{f}^{M}(\theta^{\dagger})\| = 2\|\textbf{f}(\theta^{\dagger})\|=0.
	\end{equation*}
	Thus, there exists $L_2$ such that $\forall M>L_2$
	\begin{equation}\label{eq:condN2}
	\|F^{M}(\theta)\| <\frac{r}{\gamma}.
	\end{equation}
	Let $M>\max\{L_1, L_2\}=:L$, and combine conditions \eqref{eq:condN1} and \eqref{eq:condN2}, by Lemma \ref{lem:p4} to $F^{M}$ at $\theta^{\dagger}$, we conclude that $F^{M}(\theta)=0$ has a solution in $B(\theta^{\dagger}, r)$. Denote such solution as $\theta^{*}_{M}$ $\forall M>L$, which is a minimizer of the approximated least-squares problem \eqref{poly_ls}.
	
	Our next task is to show that  $\theta^{*}_{M}$ converges to $\theta^{\dagger}$. Notice that in multi-dimensional case, we do not have the mean-value theorem. As a remedy, we introduce the following matrix-valued function
	\begin{equation*}
	J(\theta; \theta^{\dagger}):= \int^{1}_{0} \nabla \textbf{f}(\theta^{\dagger}+t(\theta-\theta^{\dagger})) \td t.
	\end{equation*}
	From the definition, we know that $J(\theta^{\dagger};\theta^{\dagger})=\nabla \textbf{f}(\theta^{\dagger})$ which is invertible and
	\begin{equation}\label{eq:JN}
	\begin{split}
	\textbf{f}(\theta)-\textbf{f}(\theta^{\dagger})&= \int_{0}^{1}\nabla \textbf{f}(\theta^{\dagger}+t(\theta-\theta^{\dagger})) (\theta-\theta^{\dagger})\td t \\
	&=J(\theta;\theta^{\dagger})(\theta-\theta^{\dagger}).
	\end{split}
	\end{equation}
	Since $J(\theta; \theta^{\dagger})^{-1}$ exists at $\theta=\theta^{\dagger}$ and it is a continuous function of $\theta$ close to $\theta^{\dagger}$, we are able to select $r>0$ (the same $r$ in \eqref{eq:condN1} and \eqref{eq:condN2}) such that $J(\theta; \theta^{\dagger})^{-1}$ exists and 
	\begin{equation*}
	\|J(\theta; \theta^{\dagger})^{-1}\|< \Delta, \quad \forall \theta\in O(\theta^{\dagger},r), 
	\end{equation*}
	for some positive constant $\Delta$. Let $\theta=\theta^{*}_{M}$ in \eqref{eq:JN}, and we obtain
	\begin{equation*}
	\begin{split}
	\|\theta^{*}_{M}-\theta^{\dagger}\| &\leq 	\|J(\theta^{*}_{M}; \theta^{\dagger})^{-1}\| \|\textbf{f}(\theta^{*}_{M})-\textbf{f}(\theta^{\dagger})\| < \Delta \|\textbf{f}(\theta^{*}_{M})\| \\
	& \leq \Delta \left(\|\textbf{f}(\theta^{*}_{M})-\textbf{f}^{M}(\theta^{*}_{M})\| +\|\textbf{f}^{M}(\theta^{*}_{M})\|\right) \leq \Delta \left(\|\textbf{f}-\textbf{f}^{M}\|_{\infty}+\|\textbf{f}^{M}(\theta^{\dagger}\|\right).
	\end{split}
	\end{equation*}
	Thus, 
	\begin{equation*}
	\lim_{M\rightarrow +\infty} 	\|\theta^{*}_{M}-\theta^{\dagger}\| \leq \lim_{M \rightarrow +\infty} \Delta \left(\|\textbf{f}-\textbf{f}^{M}\|_{\infty}+\|\textbf{f}^{M}(\theta^{\dagger})\|\right) =0.
	\end{equation*}
	As for the residual error $\textbf{f}^{M}(\theta^{*}_{M})$, we know
	\begin{equation*}
	\lim_{M\rightarrow +\infty} \|\textbf{f}^{M}(\theta^{*}_{M}) \|\leq \lim_{M \rightarrow +\infty} \|f^{M}(\theta^{\dagger})\|=0.
	\end{equation*}
This concludes the proof.
\end{proof}

\section{Proof of Theorem \ref{full_rank}}\label{appB}
The full rank condition of the Jacobian matrix is usually required in the local convergence theorem of Newton-like methods. We are going to discuss this issue over a finite dimensional function space $\Gamma^{M}_{N}$ defined by
\begin{equation*}
\Gamma^{M}_{N}:=\spn\{\theta_{1}^{k_{1}}\theta_{2}^{k_{2}}\cdots \theta_{N}^{k_{N}}\;\big|\; k_{i}=0,1,\dots,M\}.
\end{equation*}
Since the zero element in $\Gamma^{M}_{N}$ is the function is identically zero, we can introduce the following definition of linear independence.
\begin{defn}\label{def1}
	$\{f_{1},\dots,f_{n}\} \subset \Gamma^{M}_{N}$ is said to be a linearly independent set if 
	\begin{equation*}
	\sum_{i=1}^{n} c_{i}f_{i}\equiv 0 \Rightarrow c_{i}=0,\;\; \forall i=1,\dots ,n.
	\end{equation*}
\end{defn}
Define the differential operator $D_{\theta_{i}}:\Gamma^{M}_{N}\longrightarrow \Gamma^{M}_{N}$ as follows,
\begin{equation}
D_{\theta_{i}}f:=\frac{\partial}{\partial \theta_{i}}f, \quad i=1,\dots N.\nonumber
\end{equation}
Now let's deduce some relevant properties of such operator.
\begin{lem}
	The operators $\{D_{\theta_{i}}\}_{i=1}^{N}$ satisfy 
	\begin{enumerate}
		\item 
		$\rank(D_{\theta_{i}})=M(M+1)^{N-1}$;
		\item
		$\forall i\not=j$, there exists a permutation $P_{ij}$ satisfying $D_{\theta_{j}}=P_{ij}^{\top}D_{\theta_{i}}P_{ij}$.
	\end{enumerate}
\end{lem} 
\begin{proof}
	We first prove $\rank(D_{\theta_{1}})=M(M+1)^{N-1}$. It is enough to show that the dimension of the null space of $D_{\theta_{1}}$ is $(M+1)^{N-1}$, that is, $\dim(\ker(D_{\theta_{1}}))=(M+1)^{N-1}$. With the help of multi-indices, one  can see  that $\{ \theta^{\vec{k}} \;| \; \|\vec{k}\|_{\infty}\leq M \}$ provides a natural basis for $\Gamma^{M}_{N}$, and
	\begin{equation*}
	D_{\theta_{1}}( \theta^{\vec{k}}) \equiv 0 \Leftrightarrow k_{1}=0,
	\end{equation*} 
	which yields $\dim(\ker (D_{\theta_{1}}))=(M+1)^{N-1}$. Further notice that any 2-cycle $(i,j)$ in the symmetric group $S_{N}$ induces a permutation $P_{ij}$ over the basis,
	\begin{equation*}
	P_{ij}: \cdots \theta_{i}^{k_{i}}\cdots \theta_{j}^{k_{j}}\cdots \rightarrow \cdots \theta_{j}^{k_{j}}\cdots \theta_{i}^{k_{i}}\cdots.
	\end{equation*}
	This permutation leads to the identity $D_{\theta_{j}}=P_{ij}^{\top}D_{\theta_{i}}P_{ij}$, which implies that $\rank(D_{\theta_{i}})=\rank(D_{\theta_{1}})=M(M+1)^{N-1}$ for $i=1,\dots,N$.
\end{proof}
The following Lemma is will also be useful to prove our main result.
\begin{lem} \label{Lemma_appendix}
	Let $\sum_{i=1}^{N} \lambda_{i}D_{\theta_{i}}$ be a non-trivial linear combination of $\{D_{\theta_{i}}\}^{N}_{i=1}$, then,
	\begin{equation*}
	\rank\left(\sum_{i=1}^{N} \lambda_{i}D_{\theta_{i}}\right)\geq M(M+1)^{N-1}.
	\end{equation*}
\end{lem}

\begin{proof}
	Let $\Lambda=\sum_{i=1}^{N} \lambda_{i}D_{\theta_{i}}$ be a non-trivial linear combination of $\{D_{\theta_{i}}\}^{N}_{i=1}$, i.e., $\exists \lambda_{j}\not =0$. Since $\Lambda$ is also a linear operator over $\Gamma^{M}_{N}$, we have
	\begin{equation}
	\begin{split}
	\rank(\Lambda)&=\dim(\ker (\Lambda^{*})^{\perp})=(M+1)^{N}-\dim(\ker{\Lambda^{*}}) \\
	&=(M+1)^{N}-\dim(\ker(\Lambda)). \nonumber \end{split}
	\end{equation}
	Notice the image of $\theta^{\vec{k}}$ can be decomposed into
	\begin{equation*}
	\begin{split}
	\Lambda(\theta^{\vec{k}}) &=\lambda_{j}D_{\theta_{j}}(\theta_{1}^{k_{1}}\cdots\theta_{j}^{k_{j}}\cdots \theta_{N}^{k_{N}})+\sum_{i\not =j}\lambda_{i}D_{\theta_{i}}(\theta_{1}^{k_{1}}\cdots\theta_{j}^{k_{j}}\cdots \theta_{N}^{k_{N}}) \\
	&= \lambda_{j}k_{j} \theta_{1}^{k_{1}}\cdots\theta_{j}^{k_{j}-1}\cdots \theta_{N}^{k_{N}}+ \theta_{j}^{k_{j}}R(\theta_{1},\cdots,\theta_{j-1},\theta_{j+1},\cdots, \theta_{N}),
	\end{split}
	\end{equation*}
	where $R:=\sum_{i\not =j}\lambda_{i}D_{\theta_{i}}(\theta_{1}^{k_{1}}\cdots\theta_{j-1}^{k_{j-1}}\theta_{j+1}^{k_{j+1}}\cdots \theta_{N}^{k_{N}})$ is a polynomial independent with $\theta_{j}$. It is easy to see that if $k_{j}\not=0$, the image cannot be identically zero. This implies $\ker (\Lambda)\subset \spn\{\theta^{\vec{k}},k_{j}=0\}$, thus $\rank(\Lambda)\geq M(M+1)^{N-1}$.
\end{proof}

To gain intuitions, one can consider the following example with $N=2$. One can select the following ordered basis for $\Gamma^{M}_{2}$
\begin{equation*}
\begin{split}
\mathfrak{B}:=&\{1,\;\theta_{2},\; \theta_{2}^{2},\;\dots,\; \theta_{2}^{M},\;\;\;\theta_{1},\;\theta_{1}\theta_{2},\;\dots, 
;\theta_{1}\theta_{2}^{M},\\
&\;\;\;\theta_{1}^{2},\;\cdots,\;\theta_{1}^{M}\theta_{2}^{M}\}.
\end{split}
\end{equation*}
We can derive the matrix representation of $D_{\theta_{1}}$ and $D_{\theta_{2}}$
\begin{equation*}
D_{\theta_{1}}=\begin{pmatrix}
0 & & & & \\
I & 0 & & &\\
& 2I & \ddots & & \\
& & \ddots & \ddots & \\
& & & MI & 0
\end{pmatrix}, \quad D_{\theta_{2}}=\begin{pmatrix}
E & & & & \\
& E & & &\\
&  & \ddots & & \\
& &  & \ddots & \\
& & &  & E
\end{pmatrix},
\end{equation*}
where $I\in \bb{R}^{M+1\times M+1}$ is the identity matrix and $E\in \bb{R}^{M+1\times M+1}$ is defined as follows,
\begin{equation*}
E:=\begin{pmatrix}
0 & & & & \\
1 & 0 & & &\\
& 2 & \ddots & & \\
& & \ddots & \ddots & \\
& & & M & 0
\end{pmatrix}.
\end{equation*}
The linear combination $D_{\theta_{1}}+\lambda D_{\theta_{2}}$ satisfies
\begin{equation*}
D_{\theta_{1}}+\lambda D_{\theta_{2}}=\begin{pmatrix}
\lambda E & & & & \\
I & \lambda E & & &\\
& 2I & \ddots & & \\
& & \ddots & \ddots & \\
& & & MI & \lambda E
\end{pmatrix}
\end{equation*}
and rank$(D_{\theta_{1}}+\lambda D_{\theta_{2}})\geq M(M+1)$ for $\forall \lambda\in \bb{R}$. 

We now turn to our discussion of the full rank condition of the Jacobian matrix,
\begin{equation}\label{jacob}
J(\theta):=\left(\frac{\partial f_{i}}{\partial \theta_{j}} \right)_{i=1,\dots, K,\; j=1,\dots, N}
\end{equation}
where $f_{i}\in \Gamma^{M}_{N}$ and $K\geq N$. Let $\textbf{f}$ be the column vector $(f_{1}(\theta),\dots, f_{K}(\theta))^{\top}$, then $\textbf{f}\in (\Gamma^{M}_{N})^{K}$ which is the $K-$time product space of $\Gamma^{M}_{N}$. We can rewrite the Jacobian matrix in the following way,
\begin{equation}
J(\theta)=(D_{\theta_{1}}\textbf{f},\dots,D_{\theta_{N}}\textbf{f}), \nonumber
\end{equation}
where $D_{\theta_{i}}\textbf{f}$ is defined by operating $D_{\theta_{i}}$ on each component of $\textbf{f}$. With above observation we are ready to prove Theorem~\ref{full_rank}.

\begin{proof}[Proof of Theorem~\ref{full_rank}]
	By the definition, we know that the column space of $J(\theta)$ is generated by 
	\begin{equation}
	\{D_{\theta_{1}}\textbf{f},\dots,D_{\theta_{N}}\textbf{f}\}\subset (\Gamma^{M}_{N})^{K}.\nonumber
	\end{equation}
	Consider a non-trivial linear combination of those vectors denoted by $\sum_{i=1}^{N}c_{i}D_{\theta_{i}}\textbf{f}$, then we have,
\begin{equation}
	\begin{split}
	\sum_{i=1}^{N}c_{i}D_{\theta_{i}}\textbf{f}\equiv0 
	&\Leftrightarrow \left(\sum_{i=1}^{N}c_{i}D_{\theta_{i}}\right)\textbf{f}\equiv 0\\
	& \Leftrightarrow \{f_{i}\}_{i=1}^{K}\subset \ker \left(\sum_{i=1}^{N}c_{i}D_{\theta_{i}}\right). \nonumber
	\end{split}
	\end{equation} 
	By Lemma \ref{Lemma_appendix} we know that the dimension of $\ker \left( \sum_{i=1}^{N}c_{i}D_{\theta_{i}}\right)$ is less or equal than $(M+1)^{N-1}$, while $\dim(\text{span}(F))> (M+1)^{N-1}$. Therefore the column space of $J(\theta)$ has full rank in $(\Gamma^{M}_{N})^{K}$. Furthermore, this implies that $p(\theta):=\det(J^{\top}(\theta)J(\theta))$ is not a zero polynomial. Since $N(J)$ corresponds with the root of $p(\theta)$, by standard results in algebra, we know that $\mathcal{N}(J)$ is a finite set for $N=1$ and a nowhere dense set for $N\geq 2$.
\end{proof}

{ 
\section{Proof of the Ergodicity of the Langevin Dynamics \eqref{Lan_sys}}\label{appC}

With a change of variables ($q = a(x-x_0)$, $p = a v$) we reduce \eqref{Lan_sys} into
\begin{equation}\label{Lan_sys_red}
\begin{cases}
&\dot{q} = p\\
&\dot{p} = -U_{0}'(q) -\gamma p + \sqrt{2\gamma k_{B}T} \dot{W},
\end{cases} \quad U_{0}(q) = \epsilon(e^{-2q}-2e^{-q}+0.01q^2).
\end{equation}
To verify the ergodicity of \eqref{Lan_sys_red}, we apply Theorem 3.2 in \cite{Mattingly:02} which requires the potential $U_0$ satisfying the following two conditions
\begin{enumerate}
	\item $U_{0}(q)\geq 0$ for all $q\in \mathbb{R}$.
	\item There exists an $\alpha>0$ and $\beta\in (0,1)$ such that 
\begin{equation}\label{eq:cond_2}
	\frac{1}{2} U'_{0}(q)q \geq \beta U_{0}(q) + \gamma^{2} \frac{\beta (2-\beta)}{8(1-\beta)}q^2 - \alpha.
\end{equation}
\end{enumerate}
For the first condition, simply notice that adding a constant term in $U_{0}$ does  not affect the dynamic or the equilibrium distribution. Replacing $U_{0}$ in \eqref{Lan_sys_red} by
\begin{equation}\label{eq:potential}
 U_{0}(q) = \epsilon(e^{-2q}-2e^{-q}+0.01q^2+1),
 \end{equation}
we have
\begin{equation*}
\epsilon(e^{-2q}-2e^{-q}+0.01q^2+1) = \epsilon[(e^{-q}-1)^2+0.01q^2]\geq 0.
\end{equation*}
For the second condition, substituting $U_{0}$ in \eqref{eq:cond_2} by \eqref{eq:potential} the inequality becomes $\forall q \in \mathbb{R}$
\begin{equation*}
-qe^{-2q} + qe^{-q} - \beta e^{-2q} + 2\beta e^{-q}+\left[0.01(1-\beta)-\frac{\gamma^2\beta(2-\beta)}{8\epsilon(1-\beta)}\right]q^2 \geq -\alpha+\beta.
\end{equation*}
It is enough to show that  
\begin{enumerate}
	\item the function $f(q)=-qe^{-2q} + qe^{-q} - \beta e^{-2q} + 2\beta e^{-q}$ has a lower bound, and
	\item for any given $\gamma,\epsilon>0$, $\exists$ $\beta\in (0,1)$ such that
	\begin{equation}\label{eq:cond_3}
		0.01(1-\beta)-\frac{\gamma^2\beta(2-\beta)}{8\epsilon(1-\beta)}>0.
	\end{equation}
\end{enumerate}
Since $f(q)$ is continuous, by checking its limits on both sides
\begin{equation*}
\lim_{q\rightarrow +\infty} f(q) =0, \quad \lim_{q\rightarrow -\infty} f(q) = \lim_{q\rightarrow -\infty} -(q+\beta)\left[\left(e^{-q}-\frac{1}{2}\right)^2- \frac{1}{4}\right]+\beta e^{-q} = +\infty,
\end{equation*}
one can see that $f(q)$ must have a lower bound. Given $\epsilon,\gamma>0$, \eqref{eq:cond_3} can be satisfied by taking $\beta \in \left(0, 1- \sqrt{1-\frac{0.08}{0.08+\gamma^2/\epsilon}}\right)$. Thus, one can conclude that both conditions can be satisfied and the Langevin dynamics \eqref{Lan_sys} is indeed ergodic under all possible value of the parameters.

}
\nocite{*}



\end{document}